\documentclass{amsart}

\usepackage{amsmath, amsthm, amsfonts, amssymb}
\usepackage{color}
\usepackage{float}
\usepackage{graphicx}
\usepackage{verbatim}
\usepackage{wrapfig}
\usepackage[all]{xy}
\usepackage{hyperref} 
\hypersetup{colorlinks=true, linkcolor=blue, citecolor=green, filecolor=red, urlcolor=green}
\usepackage{lipsum}
\usepackage{subcaption} 

\usepackage{tikz-cd}
\usetikzlibrary{arrows}
\tikzset{>=latex}

\newcommand{\Z}{\mathbb{Z}}                     
\newcommand{\R}{\mathbb{R}}                     

\newcommand{\C}{\mathbb{C}}

\renewcommand{\sl}{{\rm sl}}

\newcommand{\F}{\mathcal{F}}
\newcommand{\U}{\mathcal U}
\renewcommand{\S}{\mathcal S}

\newcommand{\J}{\mathcal{J}}
 
\newcounter{newcounter}[section]

\numberwithin{equation}{section}
\numberwithin{newcounter}{section}
\numberwithin{figure}{section}
\numberwithin{footnote}{section}

\numberwithin{equation}{section}
\numberwithin{newcounter}{section}
\numberwithin{figure}{section}
\numberwithin{footnote}{section}

\newtheorem{corollary}[newcounter]{Corollary}
\newtheorem{definition}[newcounter]{Definition}

\newtheorem{lemma}[newcounter]{Lemma}
\newtheorem{proposition}[newcounter]{Proposition}
\newtheorem{remark}[newcounter]{Remark}
\newtheorem{theorem}[newcounter]{Theorem}

\usepackage{tikz-cd}

\usetikzlibrary{arrows}
\tikzset{>=latex}
\usepackage{pgfplots}
\pgfplotsset{compat=1.7}

\usepgfplotslibrary{fillbetween}
\usetikzlibrary{patterns}

\usepackage{pstricks-add}
\usepackage{pst-func}

\usetikzlibrary{positioning}

\usepackage{amsfonts,amsthm,amssymb,latexsym,amsmath,setspace,amscd,verbatim}

\begin{document}
 
\title[Transverse foliations for mechanical systems]{Transverse foliations for two-degree-of-freedom mechanical systems}

\author[N. V. de Paulo]{Naiara V. de Paulo}
 \address{Naiara V. de Paulo, Universidade Federal de Santa Catarina, Departamento de Matem\'atica, Rua Marechal Rondon, 880, CEP 89065-200, Bairro Salto do Norte, Blumenau SC, Brazil}
 \email {naiara.vergian@ufsc.br}

\author[S. Kim]{Seongchan Kim}
  \address{Seongchan Kim, Department of Mathematics Education, Kongju National University, Gongju 32588, Republic of Korea}
  \email {seongchankim@kongju.ac.kr}

\author[Pedro A. S. Salom\~ao]{Pedro A. S. Salom\~ao}
\address{Pedro A. S. Salom\~ao, Shenzhen International Center for Mathematics, Southern University of Science and Technology, Shenzhen, China}
\email{psalomao@sustech.edu.cn}

\author[A. Schneider]{Alexsandro Schneider}
 \address{A. Schneider, Universidade Estadual do Centro-Oeste, Rua Camargo Varela de Sá, 03, Guarapuava PR, Brazil, 85040-080}
 \email {alexsandro@unicentro.br}


 \subjclass[2020]{Primary  37J55; Secondary  53D35}

\begin{abstract} 
We investigate the dynamics of a two-degree-of-freedom mechanical system for energies slightly above a critical value. The critical set of the potential function is assumed to contain a finite number of saddle points. As the energy increases across the critical value, a disk-like component of the Hill region gets connected to other components precisely at the saddles. Under certain convexity assumptions on the critical set, we show the existence of a weakly convex foliation in the region of the energy surface where the interesting dynamics takes place. The binding of the foliation is formed by the index-$2$ Lyapunov orbits in the neck region about the rest points and a particular index-$3$ orbit.  Among other dynamical implications, the transverse foliation forces the existence of periodic orbits, homoclinics, and heteroclinics to the Lyapunov orbits. We apply the results to the H\'enon-Heiles potential for energies slightly above  $1/6$. We also discuss the existence of transverse foliations for decoupled mechanical systems, including the frozen Hill's lunar problem with centrifugal force, the Stark problem, the Euler problem of two centers, and the potential of a chemical reaction.

\end{abstract}

\maketitle
\tableofcontents

\section{Introduction and main results}\label{sec:introd}

We study the dynamics of a mechanical system
\begin{equation}\label{mec_sys}
\ddot{x}= - \nabla V(x), \;\; x\in \R^2,
\end{equation}
where the potential $V$ is smooth.  Let $x(t)$ be a solution of \eqref{mec_sys}, and let $y(t):=\dot x(t)$. Then  $(x(t),y(t))\in \R^4$ solves Hamilton's equations for the Hamiltonian 
$H(x,y)= \frac{|y|^2}{2} + V(x)$, that is
\begin{equation*}
\label{eq_ham}
\begin{aligned}
    \dot{x} =  \dfrac{\partial H}{\partial y}(x,y), \quad \quad \dot{y} =-  \dfrac{ \partial H}{\partial x}(x,y).
\end{aligned}
\end{equation*}
We fix the energy $E\in \R$ and
 study the dynamics of the three-dimensional energy surface $H^{-1}(E)$. The projection of $H^{-1}(E)$  to the $x$-plane is the Hill region  
$\Omega_E=  \{V(x) \leq E\}\subset \R^2,$  and its boundary $\partial \Omega_E$ is the zero-velocity curve. 

Let $v\in V^{-1}(0)$ be a critical point of $V.$ Then $p=(v,0)\in H^{-1}(0)$ is an equilibrium point of the Hamiltonian flow of $H$.  If $v$ is a saddle point, then $p$  is a saddle-center equilibrium, i.e.,~the linearization $J_0 \mathcal{H}(p)$  admits a pair of real eigenvalues $\pm \alpha$ and a pair of purely imaginary eigenvalues $\pm i\omega$. Here, $J_0$ is the complex matrix $$J_0=\left(\begin{array}{cc}0 &  I\\ -I & 0 \end{array} \right),$$  and $\mathcal{H}$ is the Hessian  of $H$. For every  $E>0$ sufficiently small, $H^{-1}(E)$ has a unique index-$2$ hyperbolic orbit $P_{2,E}$ in the neck region about $p,$ called Lyapunov orbit. 
For the definition of the Conley-Zehnder index of a periodic orbit, see section \ref{def:CZ}.
The projection of $P_{2, E}$ to $\Omega_E$ is a simple arc in the neck region around $v$ connecting distinct points of the zero-velocity curve $\partial \Omega_E$, i.e., $P_{2,E}$ is a brake orbit.

We are particularly interested in the case where $V$ has precisely $l\geq 1$ saddle points $v_1,\ldots,v_l\in V^{-1}(0)$, which are vertices of a simple closed curve $\partial K_0 \subset V^{-1}(0)$ bounding a disk-like compact region $K_0\subset \{V\leq 0\}$. Except for the saddles, the points of $\partial K_0$ are regular, and $V<0$ in $K_0 \setminus \partial K_0.$ In that case, $K_0$ is the projection of a singular sphere-like subset  $S_0\subset H^{-1}(0)$, with precisely $l$ singularities of saddle-center type at $p_i=(v_i,0), i=1, \ldots, l.$  As the energy $E$ changes from negative to positive, a sphere-like subset  $S_E \subset H^{-1}(E)$, $E<0$, projecting to the interior of $K_0$, gets connected to other components at $p_i$'s for $E=0$. For $E>0$ small, the neck region around each $p_i$ contains a Lyapunov orbit $P_{2, E}^i$. This orbit bounds a pair of disks,  forming a two-sphere $\S_i\subset H^{-1}(E)$ that locally separates $H^{-1}(E)$. The flow is transverse to both disks in opposite directions.  See section \ref{sec_rescaling} for a more detailed description. The energy surface $H^{-1}(E), E>0,$ thus contains a compact subset,  bounded by  $\cup_i \S_i$ and also denoted by $S_E$, which is diffeomorphic to a three-sphere with $l$ disjoint open balls removed.
The left images in Figures \ref{fig:hh_proj}, \ref{fig:frozen}, and \ref{fig:stark} represent the projections of $S_E$ to their Hill regions in concrete examples.

In general, the dynamics of $S_E$ is rather complicated, combining trajectories that escape $S_E$ through some $\S_i$ forward and/or backward in time, with invariant subsets presenting rich dynamics. We aim to study the complex dynamics of $S_E$ via transverse foliations.  

\begin{definition}\label{def:weakly_convex_foliation} Let $S_E \subset H^{-1}(E)$ be the compact subset bounded by $\cup_i\S_i$  as above.   A  weakly convex foliation adapted to $S_E$ is a singular foliation $\mathcal{F}_E$ of $S_E$ satisfying the following properties:
    \begin{itemize}
 \item[(i)] The singular set $\mathcal{P}_E$  of $\mathcal{F}_E$ consists of the Lyapunov orbits $P_{2,E}^1, \ldots, P_{2,E}^l$ and an unknotted periodic orbit $P_{3,E}$ of index $ 3$. These $l+1$ periodic orbits form the binding of $\mathcal{F}_E$ and are called binding orbits.    
 
 \item[(ii)] The complement $S_E \setminus \cup_{P \in \mathcal{P}_E}P$ is foliated by properly embedded planes and cylinders, called regular leaves of $\F_E$. They consist of: a)  $l$ pairs of (rigid) planes $U_{1,E}^j,U_{2,E}^j\subset \S_j,j=1,\dots,l,$ each pair asymptotic to $P_{2,E}^j$; b) $l$ (rigid) cylinders $V_E^j,j=1,\ldots,l,$ each one asymptotic to $P_{3,E}$ and to $P_{2,E}^j$; c) $l$ families of planes $D_{\tau,E}^j,\tau \in (0,1), j=1,\ldots,l$, all asymptotic to $P_{3,E}$. The family $D_{\tau,E}^j$ breaks onto the closure of $U_{1,E}^j\cup V^j_E$ as $\tau \to 0^+$,  and onto the closure of $U_{2,E}^{j+1} \cup V_E^{j+1}$ as $\tau \to 1^-$. Here, we use the convention $l+1 = 1$.  

\item[(iii)] All regular leaves are transverse to the flow. 
\end{itemize}
See the cases $l=1$ and $l=2$ in Figure \ref{fig:weak}.
\end{definition}

  A natural question is under which conditions $S_E$ admits a weakly convex foliation. A motivation for this question comes from the works of Albers, Fish, Frauenfelder, Hofer, and van Koert on the restricted three-body problem \cite[section 2]{albers2012global}. A weakly convex foliation with the Lyapunov orbit near the first Lagrange point as a binding orbit is expected to exist for energies slightly above the first Lagrange value. Note that a weakly convex foliation implies the existence of homoclinics and/or heteroclinics to the Lyapunov orbits,  see \cite{hofer2003finite} and also \cite{Weakconvex,dePaulo_Salomao2}.

  \begin{figure}[!ht]
    \includegraphics[width=0.95\textwidth]{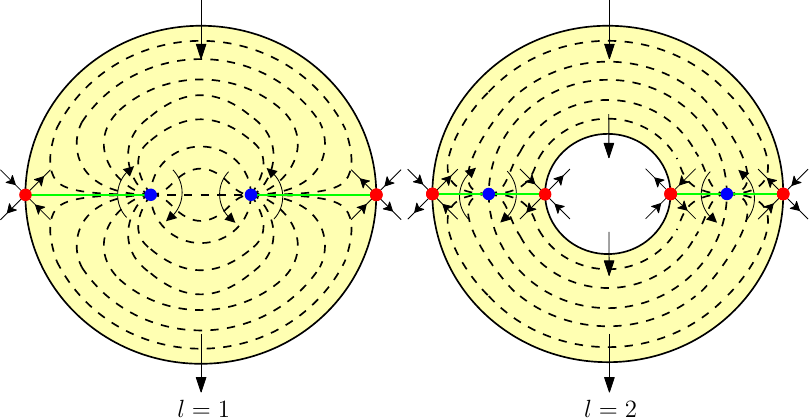}
    \caption{A cross section of a weakly convex foliation with $l=1$ (left) and $l=2$ (right). The red dots represent the Lyapunov orbits, and the blue dots represent index-$3$ orbits.  The black and green curves represent rigid planes and cylinders, respectively. The dashed curves represent the one-parameter families of planes asymptotic to the index-$3$ orbits.
    }
    \label{fig:weak}
\end{figure}
  
  Finding a weakly convex foliation is, in general, quite challenging.  Here, we point out some issues. First, it is not known for an arbitrary Hamiltonian $H$ whether the energy surface has contact type. Such a condition enables the use of $J$-holomorphic curves in symplectizations, as developed by Hofer, Wysocki, and Zehnder in \cite{Hoferinitial,hofer1996characterisation, properties_1,hhofer1999characterization,properties_3}. 
  Secondly, the compact subset $S_E\subset H^{-1}(E)$ may admit low index orbits other than the Lyapunov orbits, and these orbits may obstruct the existence of a weakly convex foliation. 

  In general, results concerning the existence of transverse foliations for Reeb flows require the periodic orbits to be non-degenerate. This is the case, for instance, in the pioneering work by Hofer, Wysocki, and Zehnder \cite{hofer2003finite} where such foliations are obtained on the tight 3-sphere. Although the non-degenericity condition is $C^\infty$-generic, it is typically very difficult to  check it in concrete examples. Here we do not assume that the flow is non-degenerate. 
  
 In \cite{dePaulo_Salomao, dePaulo_Salomao2}, the authors found weakly convex foliations assuming that the critical subset $S_0\subset H^{-1}(0)$ is strictly convex and has a unique singularity ($l=1$). Such foliations are projections of the so-called finite energy foliations in the symplectization of the energy surface. Fish and Siefring \cite{FS} studied the connected sum of finite energy foliations and constructed weakly convex foliations. See also \cite{Spatial, kim2024transverse,Lemos}. Wendl \cite{wendl2008finite} constructed finite energy foliations for overtwisted contact three-manifolds. Finite energy foliations were first studied on star-shaped hypersurfaces of $\R^4$ by Hofer, Wysocki, and Zehnder \cite{hofer1998dynamics, hofer2003finite}.    Colin, Dehornoy, and Rechtman introduced in \cite{CDR} the broken book decompositions which generalize the transverse foliations obtained by projections of finite energy foliations.

  Many concrete examples in the literature -- including most of those arising in Celestial Mechanics -- do not satisfy the geometric condition of strict convexity assumed in  \cite{dePaulo_Salomao} and, moreover, several of them present more than one saddle-center equilibrium point in the critical level. Heading toward overcoming this difficulty, the present paper extends results in \cite{dePaulo_Salomao} for $l\geq2$ under weaker assumptions on the critical set $S_0\subset H^{-1}(0)$. To explain it,  observe that $S_0$ is invariant by the flow, and its rest points $p_i=(v_i,0),i=1,\ldots,l,$ are saddle-center equilibrium points.   
   Here we assume that:
   \begin{itemize}
   \item[H1.] {\em $S_0$ is dynamically convex, i.e.,~every (non-constant) periodic orbit in $S_0$ has index at least $3$.} 
   \item[H2.] {\em $S_0$  admits a positive frame, i.e.,~the transverse linearized flow on $S_0\setminus \{p_1,\ldots,p_l\}$ strictly rotates counterclockwise on a suitable symplectic frame.}
   \end{itemize}

   Conditions H1 and H2 hold if the singular set $S_0$ is strictly convex, as proved in \cite{dePaulo_Salomao,clodoaldo_e_pedro, hofer1998dynamics}, see Remark \ref{rmk:strictlyconvex} below. These conditions may hold for more general systems, such as the H\'enon-Heiles potential, even though $S_0$ is not strictly convex. 
   
   The main result of this paper states that if $S_0$ satisfies conditions H1 and H2, then a weakly convex foliation adapted to  $S_E \subset H^{-1}(E)$ exists for every $E>0$ sufficiently small. 

\begin{theorem}\label{main_theorem} 
Let  $V$ be a real-analytic potential on $\R^2$, and let $H=|y|^2/2+V(x)$. Assume that $V$ admits precisely $l\geq 1$ saddle points $v_1,\ldots,v_l\in \partial K_0 \subset \{V= 0\}$  as above, so that the singular sphere-like subset $S_0\subset H^{-1}(0)$ projecting to  $K_0\subset \{V\leq 0\}$ 
satisfies conditions H1 and H2.  Then, for every $E>0$ sufficiently small, $H^{-1}(E)$ contains a compact subset $S_E$ admitting a weakly convex foliation $\F_E$ as in  Definition \ref{def:weakly_convex_foliation}. The binding orbits are the Lyapunov orbits $P_{2,E}^1, \ldots, P_{2,E}^l\subset \partial S_E$ and an index-$3$ periodic orbit $P_{3,E}\subset S_E \setminus \partial S_E$. Moreover, the projection $K_E\subset \R^2$ of $S_E$ to the $x$-plane converges in the Hausdorff topology to $K_0$ as $E \to 0^+$. If the actions of the Lyapunov orbits in $S_E$ coincide, then $S_E$ carries infinitely many periodic orbits, and every Lyapunov orbit has infinitely many homoclinics or heteroclinics to another Lyapunov orbit.
\end{theorem}

Conditions H1 and H2 imply the following weaker condition: 
   \begin{itemize}
       \item[H3.]  {\em For every $E>0$ sufficiently small, the energy surface $H^{-1}(E)$ carries no index-$2$ orbit near $S_0$ other than the Lyapunov orbits.}
    \end{itemize}

\begin{remark} The conclusions of Theorem \ref{main_theorem} also hold under the weaker conditions H1 and H3.   Nevertheless, with the aim of stating a result with broader applicability to concrete examples, we formulate the main theorem under assumptions H1 and H2.
\end{remark}

\begin{remark} The real-analyticity condition on the potential function $V$ in Theorem \ref{main_theorem} ensures the existence of suitable local coordinates near the saddle-center critical points of $H$, which are important for obtaining ideal index estimates in the neck regions of $S_E$ (as explored in Section \ref{sec:weak_conv}). In addition, the global real-analyticity of the Hamiltonian $H$ is relevant in \cite{Weakconvex, dePaulo_Salomao2} for deriving  dynamical consequences from the weakly convex foliation $\F_E$ when the actions of the Lyapunov orbits coincide.
\end{remark}

The proof of Theorem \ref{main_theorem} explores the fact that the energy surfaces of a mechanical system have contact type, and thus, the Hamiltonian flow is equivalent to a Reeb flow. Recall that a contact form on a three-manifold $M$ is a $1$-form $\lambda$ so that $\lambda \wedge d\lambda$ never vanishes. The Reeb vector field $R$ of $\lambda$ is the unique vector field on $M$  determined by $d\lambda(R,\cdot)=0$ and $\lambda(R)=1$. The tangent plane distribution $\xi=\ker \lambda$ is the contact structure, and the pair $(M,\xi)$ is a contact manifold.  In canonical coordinates $(x_1,x_2,y_1,y_2)\in \R^4$, the standard tight contact structure on $S^3=\{x_1^2+x_2^2+y_1^2+y_2^2=1\}$ is $\xi_0=\ker \lambda_0\subset TS^3$, where $\lambda_0$ is the contact form given by the restriction of the Liouville form $\frac{1}{2}\sum x_i dy_i - y_idx_i$ to $S^3$. Given $f\colon  S^3 \to (0,+\infty)$, the Reeb flow of $\lambda:=f\lambda_0$ is equivalent to the Hamiltonian flow on the star-shaped hypersurface  $\sqrt{f}S^3 \subset \R^4$.  We say that $\lambda$ is weakly convex if its periodic orbits have an index of at least $2$. A crucial ingredient in the proof of Theorem \ref{main_theorem} is the following key result from \cite{Weakconvex}.

\begin{theorem}[de Paulo, Hryniewicz, Kim, and  Salom\~ao \cite{Weakconvex}]\label{thm:weakconv}
Let $\lambda=f\lambda_0$ be a weakly convex contact form on the tight three-sphere $(S^3,\xi_0)$. Assume that $\lambda$ admits precisely $l\geq 1$ index-$2$ periodic orbits $P_{2,1},\ldots, P_{2,l}$, and they are all unknotted, mutually unlinked, hyperbolic, and have self-liking number $-1$. Assume that the actions of $P_{2, i}$ are smaller than that of any other periodic orbit and that every index-$3$ orbit is unlinked with any $P_{2, i}$.
Then the Reeb flow of $\lambda$ admits a weakly convex foliation $\F$ in the following sense: each $P_{2, i}$ is a binding orbit and bounds a pair of rigid planes $U_{1, i}, U_{2, i}$  whose closures form a regular embedded two-sphere $\mathcal{S}_i\subset S^3$. The remaining binding orbits are index-$3$ orbits $P_{3,j},j=1,\ldots,l+1,$ contained in distinct components $\mathcal{U}_j,j=1,\ldots,l+1,$ of $S^3 \setminus \cup_i \mathcal{S}_i$. The closure of $\U_j$ is called a chamber, and the foliation $\F$ restricted to a chamber is as in Definition \ref{def:weakly_convex_foliation}. 
\end{theorem}

\begin{remark}\label{rem_fef} The transverse foliation given in Theorem \ref{thm:weakconv} is the projection to $S^3$ of a finite energy foliation $\tilde \F$ in the symplectization $(\R \times S^3, d(e^a\lambda)),$ where $a$ is the $\R$-coordinate. An $\R$-invariant almost complex structure $J$ on $\R \times S^3$ is adapted to $\lambda$ if $J\cdot \partial_a = R$ and $J \cdot \xi_0 =\xi_0$ is $d\lambda$-compatible. The set of such $J$'s, denoted by $\J(\lambda)$, is contractible in the $C^\infty$-topology. The leaves of $\tilde \F$ are the image of embedded finite energy $J$-holomorphic curves $\tilde u=(a,u)\colon  \C P^1 \setminus \Gamma \to \R \times S^3$, where $\Gamma \subset \C P^1$ is finite, see section \ref{sec:holocurves} for definitions.  An important fact in the proof of Theorem \ref{main_theorem} is that for special choices of $J$, there exist precisely two finite energy $J$-holomorphic planes asymptotic to each $P_{2, i}$ through opposite directions. Moreover, there exists a generic subset  $\J_{\text{reg}}(\lambda) \subset \J(\lambda)$ so that if $J\in \J_{\text{reg}}(\lambda)$, then $\R \times S^3$ admits a finite energy foliation by $J$-holomorphic curves whose projection to $S^3$ is a weakly convex   foliation as in Theorem \ref{thm:weakconv}, see \cite[Theorem 5.3]{Weakconvex}.    
\end{remark}

Let us sketch out in a few brief sentences the main ideas of the proof of Theorem \ref{main_theorem}.
The potential $V$ can be modified away from an arbitrarily small neighborhood of $K_0$ in such a way that the new critical set $\hat H^{-1}(0)$ projects to the union of $K_0$ and $l$ mutually disjoint disk-like compact domains $K_1,\ldots, K_l\subset \R^2$. 
 Each $K_i$ is
connected to $K_0$ through the saddle point $v_i$ which is the unique singularity of $K_i$. For $E>0$ sufficiently small, $\hat H^{-1}(E)$ contains a regular sphere-like hypersurface $W_E$ whose projection contains $K_0\bigcup  \cup_{i=1}^l K_i$ and whose flow is equivalent to the Reeb flow of a contact form $\lambda_E$ on the tight three-sphere. 
Conditions H1 and H2 imply that we can choose $\hat H$ so that for every $E>0$ sufficiently small, $\lambda_E$ is weakly convex and the $l$ Lyapunov orbits $P_{2, E}^1,\ldots, P_{2, E}^l$ are the only index-$2$ orbits in $W_E$. A thorough study of the linearized flow near the saddle-center equilibria, using the real-analiticity of $V$, shows that the longer a periodic orbit stays near a saddle-center, the larger its index is. In particular, for $E>0$ sufficiently small, any periodic orbit in $W_E$ linking with some Lyapunov orbit must stay an arbitrarily long time near the saddle-center and thus have an arbitrarily large index. In particular, index-$3$ orbits in $W_E$ are unlinked with all the Lyapunov orbits if $E>0$ is sufficiently small. Finally, the action of each Lyapunov orbit goes to zero as $E\to 0$, and they are the shortest periodic orbits in $W_E$.  Thus, Theorem \ref{thm:weakconv} gives a weakly convex foliation $\F$ for $\lambda_E$, which is the projection of a finite energy foliation in the symplectization of $W_E$. For some special choices of almost complex structure $J$, the foliation   $\F$ admits a chamber $S_E$ bounded by all the pairs of rigid planes asymptotic to the Lyapunov orbits, which is contained in the original energy surface $H^{-1}(E)$ and projects near $K_0$. The desired weakly convex foliation given in Theorem \ref{main_theorem} is the restriction of $\F$ to $S_E$.

We apply Theorem \ref{main_theorem} to many mechanical systems. The H\'enon-Heiles potential and the frozen Hill’s lunar problem with centrifugal force meet the assumptions on the critical level and admit weakly convex foliations for energies slightly above the critical value. We also consider decoupled mechanical systems and describe the weakly convex foliations obtained from certain gradient flow lines. They include the Stark problem, the Euler problem of two centers, and a chemical reaction model.

The paper is organized as follows. In section \ref{sec_basic}, we review the basic properties of Reeb dynamics on contact-type hypersurfaces in $\R^4$ and pseudo-holomorphic curves. In section \ref{sec_proof_main}, we establish the main steps in the proof of Theorem \ref{main_theorem} as follows: in section \ref{subsec:compact}, we modify the potential away from $K_0$ to obtain, for $E>0$ sufficiently small, a sphere-like energy surface $W_E$. The weak convexity of $W_E$ and the unlinking condition between index-$2$ and index-$3$ periodic orbits are proved in section \ref{sec:weak_conv}.  In section \ref{sec:liouville}, we establish the contact property of $W_E$ and construct a particular contact form $\lambda_E$ satisfying certain suitable normal form, used in section \ref{subsec:holocurves} to define an almost complex structure on $\R \times W_E$, which admits a pair of finite energy planes asymptotic to each Lyapunov orbit and whose projections to $W_E$ are contained in the unchanged region of the potential. We then gather all these results to prove Theorem \ref{main_theorem} in section \ref{subsec_proof}. The applications are left to section \ref{sec_applications}.

\section{Basics} \label{sec_basic}

Let us introduce some basic facts about the geometry of energy surfaces in $\R^4$. 

\subsection{The quaternion frame} Let $(x_1,x_2,y_1,y_2)$ be coordinates in $\R^4$ and let $\omega_0:=dy_1 \wedge dx_1 + dy_2 \wedge dx_2$ be the standard symplectic form on $\R^4$.  Let $j_0:=I,j_1,j_2,j_3 \colon \R^4 \to \R^4$ be the orthogonal maps defined by
\begin{equation}
\label{eq_quaternion}
\begin{aligned}
j_1 \cdot (a_1,a_2,b_1,b_2) &=   (b_2,-b_1,a_2,-a_1)\\ 
j_2 \cdot (a_1,a_2,b_1,b_2) & = (a_2,-a_1,-b_2,b_1)\\ 
j_3 \cdot (a_1,a_2,b_1,b_2) & = (b_1,b_2,-a_1,-a_2) 
  \end{aligned}
\end{equation}
Let $H\colon \R^4 \to \R$ be a Hamiltonian, and let $w\in S_E:=H^{-1}(E)$ be a regular point of $H$. Let $X_0(w) = \frac{\nabla H(w)}{|\nabla H(w)|}$ be the unit vector normal to $S_E$ at $w$. The vectors
 \begin{equation}\label{eq:generalvecx}
 X_i(w):= j_i \cdot X_0(w) , \quad i=1,2,3,
\end{equation}
span the tangent space $T_wS_E$ and $X_3(w)$ is parallel to the Hamiltonian vector field $X_H$, determined by $i_{X_H} \omega_0 = -dH$. Also, $\{X_1, X_2\}$ is a symplectic basis of a plane transverse to $X_H$.

Let $w(t)\in S_E$ be a nonconstant solution of the Hamiltonian flow, that is, $w(t)$ satisfies $\dot{w}(t)=j_3 \cdot \nabla H(w(t)), \forall t$, and let $u(t)=\alpha_1(t)X_1(w(t))+\alpha_2(t)X_2(w(t)) + \alpha_3(t)X_3(w(t))$  be a linearized solution along $w(t)$,  that is, $u(t)\in T_{w(t)}S_E$ satisfies $\dot{u}(t)=j_3\cdot \mathcal{H}(w(t))u(t), \forall t$. Here, $\mathcal{H}$ denotes the Hessian of $H$. Denoting by $\alpha=(\alpha_1,\alpha_2)^T$ the transverse linearized solution in the frame $\{X_1,X_2\}$, we have
\begin{equation}\label{eq:transflow} 
\dot{\alpha} =\begin{pmatrix} 0 &  -1 \\  1 & 0 \end{pmatrix} S(t) \alpha, \quad \quad S(t): = \begin{pmatrix}   \kappa_{11}  + \kappa_{33} &  \kappa_{12} \\ \kappa_{12} &  \kappa_{22}+ \kappa_{33}    \end{pmatrix} \bigg|_{w(t)},
\end{equation}
where  $\kappa_{ij}:= \left \langle \mathcal{H} \cdot X_i, X_j \right \rangle,  i,j = 1,2,3.
 $   See, for example, \cite[Proposition D.1]{dPSsurvey}.

Let $\theta(t)$ be a continuous argument of $\alpha_1(t) + i \alpha_2(t)\in \R^+e^{i\theta(t)}$. Then
\begin{equation}\label{eq:varangle}
\begin{aligned}
\dot\theta  & = \frac{ \alpha_1 \dot{\alpha}_2 - \alpha_2 \dot{\alpha}_1}{\lvert \alpha \rvert^2} = \frac{\alpha^T S(t) \alpha }{\lvert \alpha \rvert^2}=  (\cos \theta \ \sin \theta) S(t) (\cos \theta \ \sin \theta)^T\\
& = (\kappa_{11} + \kappa_{33})\cos^2 \theta +2\kappa_{12}\cos \theta \sin \theta + (\kappa_{22} + \kappa_{33})\sin^2 \theta.
\end{aligned}
\end{equation}

\begin{lemma}\cite[Theorem 5]{clodoaldo_e_pedro}\label{lema:convexpositive}
If the Hessian $\mathcal{H}$ of  $H$ restricted to $T_{w(t)}S_E$ is positive-definite along the non-constant trajectory $w(t)\in S_E$, then a continuous argument $\theta(t)$ of any non-trivial transverse linearized solution $0\neq \alpha(t)\equiv \alpha_1(t) + i\alpha_2(t) \in \R^+e^{i\theta(t)}$ of  \eqref{eq:transflow} satisfies $\dot \theta(t) >0,  \forall t$. In particular, the frame $\{X_1, X_2\}$ is positive along $w(t)$. 
\end{lemma}

\begin{proof}
    Since $\mathcal{H}$ restricted to $T_{w(t)}S_E$ is positive-definite  along $w(t)$, we have $\kappa_{ii}>0, \forall i=1,2,3.$ Hence $\text{tr} (S(t))=\kappa_{11}+\kappa_{22}+2\kappa_{33}>0, \forall t.$ Since 
    $$\left\langle \mathcal{H}\cdot (X_1-\lambda X_2),X_1-\lambda X_2 \right \rangle|_{w(t)} = \kappa_{22}\lambda^2-2\kappa_{12}\lambda  + \kappa_{11} >0 \quad \forall \lambda\in \R,$$  we obtain $\kappa_{11}\kappa_{22}-\kappa_{12}^2>0$. This implies $\det(S(t))=\kappa_{11}\kappa_{22}-\kappa_{12}^2+\kappa_{33}(\kappa_{11}+\kappa_{22}+\kappa_{33})>0, \forall t,$ proving that $S(t)$ is positive-definite and thus $\dot \theta(t)>0, \forall t,$ see \eqref{eq:varangle}. 
\end{proof}

\begin{remark}\label{rmk:strictlyconvex}
As a consequence of Lemma \ref{lema:convexpositive}, if the critical subset $S_0 \subset H^{-1}(0)$ is strictly convex, then it admits a positive frame in the sense of condition H2 in the introduction. We say that a sphere-like subset $S_0\subset H^{-1}(0)$ admitting a finite number of saddle-centers is strictly convex if it bounds a convex subset in $\R^4$ and the Hessian $\mathcal{H}$ restricted to the tangent space at any regular point of $S_0$ is positive definite. In this particular case, it also follows from \cite[Proposition 4.7]{dePaulo_Salomao} and \cite[Theorem 3.4]{hofer1998dynamics} that $S_0$ satisfies condition H1 concerning its dynamical convexity. It turns out that if $S_0$ is strictly convex, then it satisfies the assumptions of our main result  (Theorem \ref{main_theorem}).  
\end{remark}

For a mechanical Hamiltonian $H = \frac{1}{2}|y|^2 + V(x)$, let $x\in \R^2$ be the projection of a regular point  $w\in S_E = H^{-1}(E)$. Then
\begin{equation}\label{eq_quaternionic_mechanical}
 \begin{aligned}
 X_{0} & = g \cdot (V_{ x_{1}}, V_{ x_{2}}, y_{1}, y_{2}), \quad 
 & X_{1}  = g \cdot (y_{2}, -y_{1}, V_{ x_{2}}, -V_{ x_{1}}), \\
 X_{2} & = g \cdot (V_{x_{2}}, -V_{x_ {1}}, -y_{2},  y_{1}), \quad
 & X_{3}  = g \cdot (y_{1}, y_{2}, -V_{x_ {1}}, -V_{ x_{2}}),
 \end{aligned}
 \end{equation}
 where  $V_{x_i}=\partial_{x_i}V$ and
$g =(V_{x_1}^{2}+V_{x_2}^{2}+y_{1}^{2}+y_{2}^{2})^{-1/2}.$
Then 
\begin{equation}\label{eq:kappa}
\begin{aligned}
\kappa_{11} & = g^{2}\cdot (  V_{x_1x_1}y_{2}^{2}-2V_{x_1x_2}y_{1}y_{2}+V_{x_2x_2}y_{1}^{2}+V_{x_1}^{2}+V_{x_2}^{2}), \\
\kappa_{12} & = g^{2}\cdot (  V_{x_1x_1}y_{2}V_{x_2}-V_{x_1x_2}V_{x_1}y_{1}V_{x_2}-V_{x_1x_2}y_{2}V_{x_1}\\ & \;\; \; \;+  V_{x_2x_2}y_{1}V_{x_1}-V_{x_2}y_{2}-V_{x_1}y_{1}), \\
\kappa_{22}& = g^{2}\cdot ( V_{x_1x_1}V_{x_2}^{2}-2V_{x_1x_2}V_{x_1}V_{x_2}+V_{x_2x_2}V_{x_1}^{2}+y_{1}^{2}+y_{2}^{2}), \\
\kappa_{33} & = g^{2}\cdot ( V_{x_1x_1}y_{1}^{2}+2V_{x_1x_2}y_{1}y_{2}+V_{x_2x_2}y_{2}^{2}+V_{x_1}^{2}+V_{x_2}^{2}),
\end{aligned}
\end{equation}
where  $V_{x_ix_j}=\partial_{x_ix_j}V.$
If $y_1=y_2=0$, then
\begin{equation}\label{kappaii}
\kappa_{11}=\kappa_{33}=1, \; \kappa_{12}=0 \  \mbox{ and }   \ \kappa_{22}=\frac{V_{x_1x_1}V_{x_2}^2-2V_{x_1x_2}V_{x_1}V_{x_2} + V_{x_2x_2}V_{x_1}^2}{V_{x_1}^2+V_{x_2}^2}.
\end{equation}
Hence
$\text{tr}(S)=3 + \kappa_{22}$ and $\det(S) = 2 + 2\kappa_{22}.$
In particular, the following implications are valid   
\begin{equation}\label{kappa22}
\kappa_{22}>-1 \quad \Rightarrow \quad S \mbox{ is positive-definite } \quad \Rightarrow \quad \dot \theta >0,
\end{equation}
provided $y_1=y_2=0.$
The implications above will be useful in the H\'enon-Heiles potential to show that the singular sphere-like subset of the critical energy surface admits a positive frame even though it is not strictly convex. See Proposition \ref{prop_hh2}.

\subsection{The Conley-Zehnder index}\label{def:CZ}

We can use the quaternion frame to define the index of a periodic orbit  $P=(w, T)\subset H^{-1}(E) \subset \R^4$. Let $\theta(t)$ be a continuous argument of a  non-vanishing solution $\alpha(t)\equiv \alpha_1(t) + i \alpha_2(t)$  of   \eqref{eq:transflow}.  
 Let
\begin{equation}\label{eq:vaiinterval}
I:= \left\{ \frac{\theta(T) - \theta(0)}{2\pi} \mid  \alpha(0) \neq 0 \right\}. 
\end{equation}
Then, $I$ is a compact interval, and its length is less than $1/2$. Given $\varepsilon >0$ sufficiently small, let
$I_\varepsilon := I - \varepsilon.$
Then the index of  $P$
\begin{equation*}\label{eq:defCZ}
\mu(P):=\begin{cases} 2k+1 & \text{ if } I_\varepsilon  \subset (k, k+1), \\ 2k & \text{ if } k \in {\rm int}(I_\varepsilon), \end{cases}
\end{equation*}
 is independent of $\epsilon>0$ sufficiently small.
If $S_E$ has contact-type with contact form $\lambda_E$, then $\mu(P)$ coincides with the (generalized) Conley-Zehnder index  \cite{hofer1998dynamics}, see also Long's book \cite{long2012index}. Indeed, let  $\xi \subset T S_E$  be the contact structure. Every  $v \in \xi$  can be written in the quaternion frame as $v = a_1 X_1 + a_2 X_2 + a_3X_3$ for some $a_1, a_2 , a_3 \in \R$.  Recall that $X_3$ is parallel to  $X_H.$ 
We define the isomorphism
$$
\pi_\xi \colon  \xi  \to  {\rm span}\{X_1, X_2 \}, \quad v =  \alpha_1 X_1 + \alpha_2 X_2 + \alpha_3X_3 \mapsto \alpha_1 X_1 + \alpha_2 X_2.
$$
 Set $\tilde X_j := \pi_{\xi}^{-1}(X_j),$ $j=1,2,$ so that    $\xi = \text{span}\{\tilde X_1,\tilde X_2\}$. 
 Since $X_3\subset \ker \omega_0|_{TS_E}$, we compute
$\omega_0 ( \tilde X_1, \tilde X_2) = \omega_0 ( X_1, X_2)=1.$ Hence, the frame $\{\tilde X_1,\tilde X_2\}$ induces a symplectic trivialization  $\Psi \colon \xi \to S_E \times \C$
$$
\Psi: \alpha_1 \tilde X_1 + \alpha_2 \tilde X_2 \mapsto \alpha_1+i\alpha_2,
$$
and  the linearized flow on $\xi$  is determined by \eqref{eq:transflow}. 

Denoting by $P^k=(w,kT)$ the $k$-th iterate of $P=(w,T),$ the rotation number of $P$ is the well-defined limit
$$
\rho(P) = \lim_{k\to +\infty} \frac{\mu(P^k)}{2k}.
$$

\subsection{Self-linking number}  
Let $\gamma\colon \R/ \Z \to M$ be a null-homologous transverse knot on a contact three-manifold $(M, \xi=\ker \lambda)$, i.e., $T\gamma \cap \xi =0$. Choose a Seifert surface $i\colon\Sigma \hookrightarrow M$ for $\gamma$, i.e., $\Sigma$ is an embedded oriented surface bounded by $\gamma$, and take a non-vanishing section $X$ of  $i^*\xi$. Assume that $M$ is oriented by  $\lambda \wedge d \lambda >0$, $\gamma$ is oriented by $\lambda>0$, and $\Sigma$ is oriented by  $\gamma=\partial \Sigma$. 
Let $\text{exp}$ be an exponential map on $M$. If $X$ is sufficiently small, then   $\gamma(\R/\Z) \cap \gamma_X (\R/\Z) = \emptyset,$
where  $\gamma_X(t) := \exp_{\gamma(t)}X(t), \forall t \in \R/\Z$.  The self-linking number of $\gamma$, denoted  by $\sl(\gamma)$, is defined as the algebraic intersection number  $\gamma_X\cdot i$, where $\gamma_X$ inherits the orientation of $\gamma$. Notice that $\sl(\gamma)$ is independent of $X$ and $\text{exp}$. If the Euler class $e(\xi)$ vanishes on $H_2(M)$, then $\sl(\gamma)$ is  independent of $\Sigma$.

\subsection{Pseudo-holomorphic curves in symplectizations}\label{sec:holocurves}
 
Let $(M,\xi = \ker \lambda)$ be a contact three-manifold, and let $(\mathbb{R}\times M, d(e^{a}\lambda))$ be its symplectization, where $a$ is the $\R$-coordinate. Denote by $\pi_\xi\colon TM\to  \xi$ the projection along the Reeb vector field  $R$. As above, a periodic orbit of $\lambda$ is a pair $P=(x,T)$, where $x$ is a periodic trajectory of the Reeb flow of $\lambda$ and $T>0$ is a period of $x$. We identify periodic orbits with the same image and period. Let $\mathcal{J} (\lambda)$ be the space of $\R$-invariant $d\lambda$-compatible almost complex structures $J$ on $\R \times M$ satisfying $J \cdot \partial_a = R$ and $ {J}(\xi)=\xi.$
 Let $(S, j)$ be a closed Riemann surface and $\Gamma\subset S$ a finite set. A  map $\tilde{u}=(a,u)\colon S\setminus\Gamma\rightarrow\mathbb{R}\times M$ is called a finite energy $J$-holomorphic curve if 
\begin{equation}\label{eq_psedo_holom_1}
    \bar{\partial}_{J}(\tilde{u}):=\frac{1}{2}(d\tilde{u}+J(\tilde{u})\circ d\tilde{u}\circ j)=0,
\end{equation}
and its Hofer's energy is finite
\begin{equation*}
0<E(\tilde{u}):=\sup_{\phi\in\Lambda}\int_{S\setminus\Gamma}\tilde{u}^{*}d(\phi\lambda)<\infty.
\end{equation*}
Here, $\Lambda:=\{\phi\in C^{\infty}(\mathbb{R}, [0, 1]) \mid  \phi'\geq0\}$.
If $S=\C P^1$ and $\#\Gamma=1,$ then $\tilde u$ is called a finite energy plane. Near each point $z\in S$, one may consider holomorphic coordinates $s+it$ and rewrite  \eqref{eq_psedo_holom_1} as
\begin{equation}
\label{eq_psedo_holom_2}
\begin{cases}
    \pi_\xi u_s(s,t)+J(u(s,t))\pi_\xi u_t(s,t)=0,\\
\lambda(u_t(s,t))=a_s(s,t),\\
\lambda(u_s(s,t))=-a_t(s,t).
\end{cases}
\end{equation}

The elements in $\Gamma$ are called punctures. A puncture $z_0 \in \Gamma$ is called removable if $\tilde u$ can be smoothly extended over $z_0$.    Otherwise, we say that $z_0$ is non-removable.   In the following, we assume that every puncture is non-removable. 
Given $z_0 \in \Gamma, $ we choose a neighborhood $U_0 \subset S \setminus \Gamma$ of $z_0$ and a bi-holomorphism $\phi_{z_0} \colon ( \mathbb{D}, 0) \to (U_0,z_0)$ such that $\phi_{z_0}^* j = i$, where $\mathbb{D} \subset \C$ is the closed unit disk centered at $0$ and $i$ is the canonical complex structure on $\C$. In cylindrical coordinates $ [0, +\infty) \times \R /  \Z \ni (s,t) \mapsto \phi_{z_0}( e^{ -2\pi (s+it)})$ near $z_0$, we write 
$
\tilde u(s,t) = ( a(s,t), u(s,t)) = \tilde u \circ \phi_{z_0} ( e^{ -2\pi (s+it)}). $

The following result is due to Hofer establishing a deep connection between finite energy $J$-holomorphic curves and periodic Reeb orbits.

\begin{theorem}[Hofer \cite{Hoferinitial}]\label{thm_hofer}
Let $(s, t)\in [0,+\infty) \times \R / \Z$ be  cylindrical coordinates near a puncture $z_0\in \Gamma$ of a finite energy curve $\tilde u = (a,u) \colon S \setminus \Gamma \to \R \times M.$ Write $\tilde u(s,t)=(a(s,t),u(s,t))$ in cylindrical coordinates as above. Given a  sequence $s_{n}\rightarrow+\infty$, there exists a subsequence, still denoted by $s_n$, and a periodic orbit $P=(w, T)$ such that
$u(s_n, \cdot )\to w(\varepsilon  T\cdot)$ in $C^{\infty} (\R/ \Z,M)$ as $n \to +\infty.$
The sign $\varepsilon=\varepsilon(z_0) \in \{ \pm1  \}$ depends only on $z_0$.
\end{theorem}

A non-removable puncture $z_0 \in \Gamma$ is said to be positive or negative if $\varepsilon(z_0)=+1$ or $\varepsilon(z_0)=-1,$ respectively. Then $a(z)\to \pm \infty$ as $z\to z_0$ according to the sign $\epsilon(z_0) = \pm 1$. A periodic orbit $P=(w, T)$ as in Theorem \ref{thm_hofer} is an asymptotic limit of $\tilde u$ at $z_0$. The set of asymptotic limits of $\tilde u$ at $z_0$ is compact and connected; see \cite[Lemma 13.3.1]{FvKbook} and \cite{Si3}.

The following theorem of Hofer, Wysocki, and Zehnder shows that if an asymptotic limit of $\tilde{u}$ at $z_0$ is nondegenerate, then it is the unique asymptotic limit of $\tilde{u}$ at $z_0$. 

\begin{theorem} [Hofer-Wysocki-Zehnder \cite{properties_1}] 
Let  $z_0\in \Gamma$ be a puncture of a finite energy curve $\tilde u = (a, u) \colon S \setminus \Gamma \to \R \times M$ and let   $P=(w, T)$ be an asymptotic limit of $\tilde u$ at $z_0.$ If $P$ is non-degenerate, then  $P$ is the unique asymptotic limit at $z_0$.  
\end{theorem}

\subsection{Re-scaled coordinates near a saddle-center}\label{sec_rescaling} Throughout the paper we shall  re-scale coordinates near the saddle-center $p=(v,0)\in H^{-1}(0)$, where $v$ represents any saddle point $v_1, \ldots,v_l$ of $V.$ We may assume that $v=0$  and $V=ax_1^2/2 + bx_2^2/2 + R(x),$ where $a<0$, $b>0$ and $R(x)=O(|x|^3).$ Taking new coordinates $(x,y)=\sqrt{\epsilon}( \hat x,\hat y)$, with $\epsilon>0$ small, we consider the Hamiltonian 
$$
\hat H_\epsilon(\hat x,\hat y) = \frac{1}{\epsilon}H(x,y) =\frac{\hat y_1^2+\hat y_2^2}{2}+ \frac{a\hat x_1^2}{2} + \frac{b\hat x_2^2}{2}+ \frac{1}{\epsilon} \hat R_\epsilon(\hat x),
$$ 
where $\hat R_\epsilon(\hat x) = R(\sqrt{\epsilon}\hat x)$. The rescaled potential $V(\sqrt{\epsilon} \hat x)/\epsilon$ is denoted by $\hat V_\epsilon(\hat x)$. Notice that $H^{-1}(\epsilon E)$ corresponds to $\hat H_\epsilon^{-1}(E), \forall E\in \R$. We fix $E>0$. Since $\hat R_\epsilon(\hat x)/\epsilon$  converges in $C^\infty_{\text{loc}}(\R^2)$ to $0$ as $\epsilon \to 0^+,$
 $\hat H_\epsilon$ converges in $C^\infty_{\text{loc}}(\R^4)$ to 
$$
\hat H_0(\hat x,\hat y):=\frac{\hat y_1^2+\hat y_2^2}{2} + \frac{a\hat x_1^2}{2} + \frac{b\hat x_2^2}{2}
$$
as $\epsilon \to 0^+$. The potential of $\hat H_0$ is denoted by $\hat V_0(\hat x) = a\hat x_1^2/2+b\hat x_2^2/2$. The equilibrium point $(\hat x, \hat y)=(0,0)$ of $\hat H_0$ projects into the $\hat x_1 \hat y_1$-plane as a saddle and into the $\hat x_2 \hat y_2$-plane as a center.
The orbit $\hat P_{2,0,E}=\{\hat x_1=\hat y_1=0, b\hat x_2^2+\hat y_2^2=2E\}\subset \hat H_0^{-1}(E)$ is an index-$2$ hyperbolic orbit. It is the limit as $\epsilon \to 0^+$ of index-$2$ hyperbolic orbits  $\hat P_{2,\epsilon, E} \subset \hat H_\epsilon^{-1}(E)$ corresponding to the Lyapunov orbits  $P_{2,\epsilon E} \subset H^{-1}(\epsilon E)$ near $0$. Notice that both the disks
$$
\hat U_{1,0,E}:=\hat H_0^{-1}(E) \cap \{\hat x_1=0,\hat y_1>0\} \; \mbox{ and } \; \hat U_{2,0,E}:=\hat H_0^{-1}(E) \cap \{\hat x_1=0, \hat y_1<0\},$$
have $\hat P_{2,0,E}$ as boundary and are transverse to the flow on $\hat H_0^{-1}(E)$. See Figure \ref{fig:lyapunov}. The behavior of $\hat U_{1,0,E}$ and $\hat U_{2,0,E}$ near $\hat P_{2,0,E}$ and the $C^\infty_{\text{loc}}$-convergence of $\hat H_\epsilon$ to $\hat H_0$ imply that $\hat P_{2,\epsilon,E}$ bounds embedded disks $\hat U_{1,\epsilon,E}, \hat U_{2,\epsilon,E}\subset \hat H_\epsilon^{-1}(E)$ transverse to the flow and $C^\infty$-close to $\hat U_{1,0,E}$ and $\hat U_{2,0,E}$, respectively, for every $\epsilon>0$ sufficiently small. In particular, if $Q_{\epsilon E} \subset H^{-1}(\epsilon E)$ is a periodic orbit linked with $P_{2,\epsilon E}$, then $\hat Q_{\epsilon, E}:=Q_{\epsilon E}/ \sqrt{\epsilon} \subset \hat{H}^{-1}_{\epsilon}(E)$ is linked with $\hat P_{2,\epsilon,E}$ and therefore intersects $\hat U_{1,\epsilon,E}$ and $\hat U_{2,\epsilon,E}.$ 

We shall see later that such a linked orbit $Q_{\epsilon E}$ has an arbitrarily high index if $\epsilon, E>0$ are sufficiently small. Also, we shall choose the disks $\hat U_{1,\epsilon, E}, \hat U_{2,\epsilon, E}$ as projections of finite energy planes in the symplectization for suitable almost complex structures.

\begin{figure}[!ht]
    \includegraphics[width=0.7\textwidth]{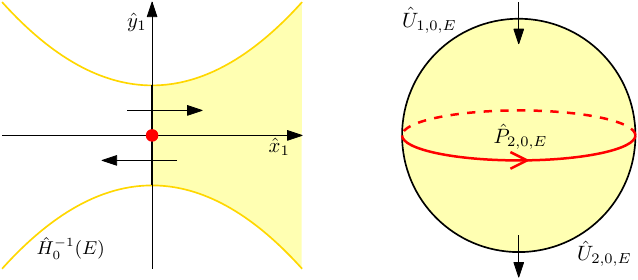}
    \caption{Projections of the disks $\hat U_{1,0,E}$ and $\hat U_{2,0,E}$ to the $\hat x_1\hat y_1$-plane represented in bold black lines. The flow goes across the disks in different directions (left). The hyperbolic orbit $\hat P_{2,0,E}$ is the equator of a two-sphere that locally separates the energy level and mediates the entry/exit of the flow into/from a chamber (right).}
    
    \label{fig:lyapunov}
\end{figure}


\section{Proof of Theorem \ref{main_theorem}}\label{sec_proof_main}

 The proof of Theorem \ref{main_theorem} is organized in several steps:
 \begin{itemize}
     \item We start by changing the Hamiltonian $H$ away from $S_0$ to cap the critical set near the saddle-centers. To do that, we change the potential $V$ away from $K_0$ so that the Hill region $\{V\leq 0\}$ becomes the union of $K_0$ and $l$ disjoint compact disks $K_1,\ldots, K_l$, touching $K_0$ at the saddles $v_1,\ldots,v_l$. The new Hamiltonian has the same saddle-centers $p_1,\ldots,p_l\in H^{-1}(0)$ and, for $E>0$ sufficiently small, $H^{-1}(E)$ is a sphere-like hypersurface $W_E\subset \R^4$ whose projection contains $K_0\bigcup \cup_{i=1}^l K_i$. This is the content of section \ref{subsec:compact}.
     \item In section \ref{sec:weak_conv}, we show that $W_E$ is weakly convex, and the Lyapunov orbits are the only index-$2$ orbits for every $E>0$ sufficiently small. Moreover, we study the linearized flow near a saddle-center and show that index-$3$ orbits in $W_E$ cannot be linked with any Lyapunov orbit if $E>0$ is sufficiently small. 
     \item A special contact form $\lambda_E$ on $W_E$ is constructed in section \ref{sec:liouville}. This contact form is crucial in the location of rigid planes asymptotic to the Lyapunov orbits. 
     \item In section \ref{subsec:holocurves}, we choose a special almost complex structure $J_E$ in the symplectization of $W_E$, which is adapted to $\lambda_E$ and admits a pair of rigid planes asymptotic to each Lyapunov orbit. The projections of the rigid planes to $W_E$ lie in the unchanged domain of the Hamiltonian $H$ for every $E>0$ sufficiently small.  
     \item The weakly convex foliation in the domain of $H^{-1}(E)$ determined by the rigid planes then follows from Theorem \ref{thm:weakconv} and the previous steps. See Remark \ref{rem_fef}.
     
 \end{itemize}

\subsection{Capping the energy surface}\label{subsec:compact} 
 Recall that $K_0\subset \{V\leq 0\}$ is the compact subset given by the projection to the $x$-plane of the singular sphere-like subset $S_0\subset H^{-1}(0)$ containing $l$  saddle-centers of $H$ which project to $l$ saddle points of $V$ in the boundary of $K_0$.  Since the energy surface $H^{-1}(E), E>0$ small, may not be compact at first, it will be convenient to change the potential $V$ away from $K_0$ so that $H^{-1}(E)$ becomes a sphere-like regular component projecting near $K_0$, for every $E>0$ small.

We may assume that $0\in \partial K_0$ is one of the saddles of $V$ and $V(x) = ax_1^2/2 + bx_2^2/2 + R(x)$ near $0$, where $a<0$, $b>0$ and $R(x)=O(|x|^3)$. We can also assume that $K_0$ is locally contained in $\{x_1 \leq 0\}$. We  locally change $V$ on $\{x_1>0\}$. 
For each $\epsilon>0$ small, consider the re-scaled Hamiltonian $\hat H_\epsilon$ and re-scaled potential $\hat V_\epsilon$ as in section \ref{sec_rescaling}. Choose a smooth function $f\colon(-\infty,2) \to [0,+\infty)$ so that 
\begin{equation}\label{eq_cutfunction}
f\equiv 0 \mbox{ on }(-\infty,1], \ \ f'''>0 \mbox{ on }(1,2), \mbox{ and }\displaystyle\lim_{t\to 2^-} f(t) =+\infty,
\end{equation}
and let
\begin{equation}\label{Vhat_epsilon}
\tilde V_\epsilon(\hat x):= \hat V_\epsilon(\hat x) + f(\hat x_1), \quad \forall \hat x = (\hat x_1,\hat x_2) \; \; \mbox{with} \; \; \hat x_1 < 2.
\end{equation}
Note that $\tilde V_\epsilon$ coincides with $\hat V_\epsilon$ on $\{\hat x_1 \leq 1\}$ and thus $0\in\R^2$ is a saddle point of $\tilde V_\epsilon$.

\begin{lemma}\label{lem:sl}
    For every $\epsilon>0$ sufficiently small, $\{\tilde V_\epsilon \leq 0\}$ contains a disk-like compact set $\tilde K_{\epsilon} \subset \{0\leq \hat x_1 < 2\}$ with a unique singularity at $0$. Except for the origin, every point in  $\partial \tilde K_\epsilon$ is a regular point of $\tilde V_\epsilon,$ and $\tilde V_\epsilon <0$ in $\tilde K_\epsilon \setminus \partial \tilde K_\epsilon.$
\end{lemma}

\begin{proof}Notice that $\tilde V_\epsilon$ converges in $C^\infty_{\text{loc}}(\{\hat x_1<2\})$ to $\tilde V_0(\hat x_1,\hat x_2):=\hat V_0(\hat x_1,\hat x_2) +f(\hat x_1) = a\hat x_1^2/2+b\hat x_2^2/2 + f(\hat x_1)$ as $\epsilon\to 0^+$. 
By \eqref{eq_cutfunction},  $g(\hat x_1):=a \hat x_1^2/2+f(\hat x_1)$ has a unique critical point $\hat q_1$ in the interval $(0,2),$ which is a nondegenerate minimum.  
Hence $\tilde V_0$ has a unique nondegenerate minimum  $\tilde q =(\hat q_1,0)$ in $\{0< \hat x_1<2\}$ with $\tilde V_0(\tilde q )=g(\hat q_1)<0$. For every  $E\in (\tilde V_0(\tilde q ),0)$, $\tilde V_0^{-1}(E)$ contains a regular circle-like level on $\{0<\hat x_1<2\}$. Such a family of circles develops a singularity at $0\in \R^2$ as $E \to 0^-$, and thus $\tilde V_0^{-1}(0)$ contains a singular circle-like subset $\partial \tilde K_0$ with a unique singularity at $0$, bounding a disk-like compact subset $\tilde K_0 \subset \{\tilde V_0 \leq 0\}\cap \{0\leq \hat x_1 <2\}$. 

Since both singularities at $0$ and $\tilde q $ are nondegenerate critical points of $\tilde V_0$, similar conclusions hold for $\tilde V_\epsilon$, for every $\epsilon>0$ sufficiently small, that is $\tilde V_\epsilon^{-1}(0)$ contains a singular circle-like subset $\partial \tilde K_\epsilon$ with a unique singularity at $0$, bounding a disk-like compact subset $\tilde K_\epsilon \subset \{\tilde V_\epsilon \leq 0\}$ and $\tilde V_\epsilon|_{\tilde K_\epsilon \setminus \partial \tilde K_\epsilon}<0$.  This finishes the proof.
\end{proof}

In the original coordinates $x=(x_1,x_2)$, the compact set $\tilde K_\epsilon$ corresponds to $K_\epsilon:=\sqrt{\epsilon} \tilde K_\epsilon$ and can be made arbitrarily close to the saddle point. Performing this construction near each saddle $v_1,\ldots,v_l$ of $V$ we end up with $l$ disk-like compact sets $K_{i,\epsilon}:=\sqrt{\epsilon} \tilde K_{i,\epsilon}, i=1,\ldots,l$, each one admitting a unique singularity at the corresponding $v_i$. 
The following proposition follows directly from Lemma \ref{lem:sl}.

\begin{proposition}\label{prop_capping} Let $\U_0 \subset \R^2$ be an open neighborhood of the disk-like compact set $K_0\subset \R^2.$ For every $\epsilon>0$ sufficiently small, there exists a potential $V_\epsilon\colon\R^2 \to \R$ coinciding with $V$ near $K_0$ so that
\begin{itemize}  
     \item[(i)] The Hill region $\{V_\epsilon\leq 0\}$ contains a subset formed by $K_0$ and the union of  disjoint  disk-like compact sets $K_{1,\epsilon},\ldots,K_{l,\epsilon}\subset \U_0,$ touching $K_0$ precisely at the saddles $v_1,\ldots,v_l$.  The saddles $v_1,\ldots,v_l$ are the unique critical points of $V$ in $V^{-1}(0)$.
     
    \item[(ii)] For every $E>0$ sufficiently small, the Hamiltonian $H_\epsilon(x,y):=|y|^2/2 +  V_\epsilon(x)$ admits a regular sphere-like component $W_{\epsilon,E}=  H_\epsilon^{-1}(E)$ projecting to a disk-like compact set $K_{\epsilon,E} \supset K_0 \bigcup \cup_{i=1}^l K_{i,\epsilon}$.  See Figure \ref{fig:compact}.

    \item[(iii)] In re-scaled coordinates $(\hat x,\hat y)=(x/\sqrt{\epsilon},y/\sqrt{\epsilon})$ near the saddle-centers, the Hamiltonian $ H_\epsilon(\sqrt{\epsilon}\hat x,\sqrt{\epsilon}\hat y)/\epsilon$ converges in $C^\infty_{\text{loc}}(\{\hat x_1 <2\})$ to 
    $$
    \tilde H_0(\hat x,\hat y):=\frac{|\hat y|^2}{2} + \frac{a\hat x_1^2}{2} + \frac{b\hat x_2^2}{2} + f(\hat x_1), 
    $$ 
    as $\epsilon\to 0^+$, where $a<0, b>0$, and $f$ satisfies \eqref{eq_cutfunction}.
\end{itemize}
 
\end{proposition}

\begin{proof} 
 The new potential $V_\epsilon$ near each saddle $v_1,\ldots,v_l$ is obtained in local coordinates as in Lemma \ref{lem:sl} by defining $V_\epsilon(x) := \epsilon \tilde V_\epsilon(x/\sqrt{\epsilon})$. Locally, each  compact subset of $\{ V_\epsilon \leq 0\}$  is given by   $K_{i,\epsilon}=\sqrt{\epsilon} \tilde K_{i,\epsilon}$. Hence $K_{i,\epsilon}$ lies in $\U_0$ if 
 $\epsilon>0$ is taken sufficiently small. We may assume that $V_\epsilon >0$ on $\R^2 \setminus (K_0\bigcup \cup_{i=1}^l K_{i,\epsilon})$. For every $\epsilon>0$ fixed sufficiently small, $V_\epsilon$ coincides with $V$ near $K_0$ and the projection $K_{\epsilon,E}$ of the sphere-like component $W_{\epsilon,E}:= H_\epsilon^{-1}(E)$ is contained in $\U_0$ for every $E>0$  sufficiently small.
\end{proof}

 \begin{figure}[!ht]
    \includegraphics[width=0.8\textwidth]{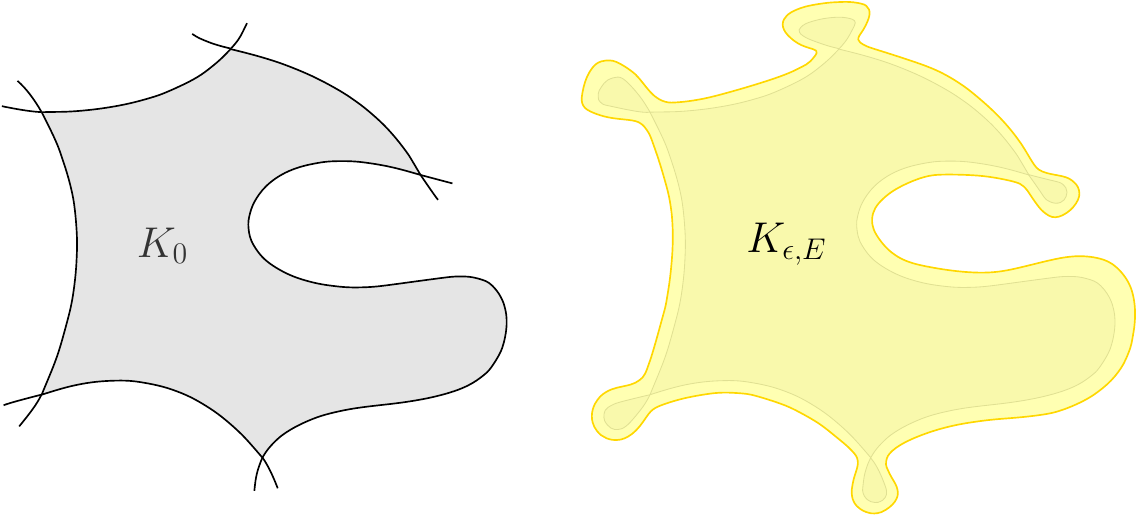}
    \caption{The compact subset $K_0\subset \{V\leq 0\}$ before changing the potential (left). The Hill region $K_{\epsilon,E}=\{V_\epsilon \leq E\}$ after capping the critical set (right), with $\epsilon,E>0$  small.}
    \label{fig:compact}
\end{figure}

\subsection{Weak convexity}\label{sec:weak_conv}
Consider the Hamiltonian $H_\epsilon=|y|^2/2+V_\epsilon(x)$ as in Proposition \ref{prop_capping}, where $\epsilon>0$ is sufficiently small. 
Let $W_{\epsilon, E}= H_\epsilon^{-1}(E),$ $E>0$ small,  be the regular sphere-like hypersurface whose projection to the $x$-plane contains $K_0 \bigcup \cup_{i=1}^l K_{i,\epsilon}$. Our goal in this section is to prove the following proposition.

\begin{proposition}\label{prop_weakconv}
 Assume that the Hamiltonian $H$ is real-analytic and the singular sphere-like subset $S_0$ satisfies conditions H1 and H2 as stated in the introduction.
 For every $\epsilon>0$ sufficiently small, the following holds: if $E>0$ is sufficiently small, then
\begin{itemize}
    \item[(i)]  The Hamiltonian flow  of $H_\epsilon$  on $W_{\epsilon,E}$ is weakly convex.

    \item[(ii)] The Lyapunov orbits in $W_{\epsilon,E}$ around the saddle-centers $p_1,\ldots,p_l$ are the only index-$2$ periodic orbits in $W_{\epsilon,E}$.
    
    \item[(iii)] If $P\subset W_{\epsilon,E}$ is an index-$3$ orbit, then $P$ is not linked with any Lyapunov orbit. 

    \item[(iv)] The Lyapunov orbits have the smallest actions among the actions of all periodic orbits in $W_{\epsilon, E}$.
    \end{itemize}
\end{proposition}

Before proving Proposition \ref{prop_weakconv}, we study the linearized flow of a general Hamiltonian $H:\R^4 \to \R$, not necessarily mechanical, near a saddle-center. We follow \cite{dePaulo_Salomao} and assume that $H$ is real-analytic near the saddle-center. 

\begin{lemma}\label{lem_deltatheta}
    Let $0\in \R^4$ be a saddle-center of a real-analytic Hamiltonian $H=H(x,y)$. Let $U_0 \subset \R^4$ be a small open neighborhood of $0$, and let $\{X_1, X_2\}$ be the transverse frame induced by the quaternions, see \eqref{eq:generalvecx}, defined on $\text{closure}(U_0)\setminus \{0\}$. Given $M>0$, there exist compact neighborhoods $U_M \subset U_1$ of $0$ contained in the interior of $U_0$ with the following significance. Let $0<\lambda\leq 1$, and let $\gamma(t), t\in [-T,T],$ be a trajectory of $H$ satisfying $\gamma(t) \in (\lambda U_1)\setminus \partial (\lambda U_1), \forall t\in (-T,T),$ $\gamma(\pm T) \in  \partial (\lambda U_1),$ and $\gamma([-T,T]) \cap (\lambda U_M) \neq \emptyset.$ Let $\theta (t)$ be a continuous argument of a non-trivial linearized solution along $\gamma$ in the frame  $\{X_1,X_2\}$, see \eqref{eq:transflow}. Then $\theta(T) - \theta(-T) > M.$
\end{lemma}

\begin{proof}
Since $H$ is real-analytic near $0$,  there exist real-analytic symplectic coordinates $(q,p)=\varphi(x,y)$ near $0\in \R^4$ so that $H$ (or $-H$) takes  the form
\begin{equation}\label{eq:KMoser}
K=-\alpha I_1 + \omega I_2 + R(I_1,I_2),
\end{equation}
where $\alpha,\omega>0,$ $I_1=q_1p_1$, $I_2=(q_2^2+p_2^2)/2$ and $R(I_1,I_2)=O(I_1^2+I_2^2)$.
These coordinates are due to J. Moser \cite{moser_coordinates} and H. R\"ussmann \cite{russmann_refinement}. 
The symplectic form in the new coordinates $(q,p)$ is $\sum_i dp_i \wedge dq_i$ and the equations of motion are
$$
\left\{
\begin{aligned}
    \dot q_1 &  = -\bar \alpha q_1 \\
    \dot p_1 &  = \bar \alpha p_1
\end{aligned}\right. \hspace{1.5cm}
\left\{
\begin{aligned}
    \dot q_2 &  =\bar \omega p_2 \\ 
   \dot p_2 & = -\bar \omega q_2
\end{aligned} \right.
$$
where $\bar \alpha:= \alpha - \partial_{I_1}R$ and $\bar \omega: = \omega + \partial_{I_2}R$. Since $I_1$ and $I_2$ are preserved by the flow, $\bar \alpha$ and $\bar \omega$ do not depend on time, and the solutions are
\begin{equation}\label{solutions}
\begin{aligned}
    q_1(t)  = q_1(0)e^{-\bar \alpha t}, \quad p_1(t)  = p_1(0)e^{\bar \alpha t},\\
    q_2(t) + i p_2(t) = e^{-i \bar \omega t} (q_2(0)+ip_2(0)).
\end{aligned}
\end{equation}


 Let $\delta>0$ be small so that in coordinates $(q,p)$ we have $B_\delta(0) \subset \varphi(U_0)$. Let $0<\lambda \leq 1$. Consider a solution $\gamma(t) = (q_1(t),q_2(t),p_1(t),p_2(t))\in  B_{\lambda \delta}(0)$ satisfying 
\begin{equation}\label{signb}
\begin{aligned}
 \quad q_1(0)=p_1(0) = \lambda b>0 \quad &\mbox{ or } \quad q_1(0) = -p_1(0) = \lambda b>0,\\
  \mbox{and }q_2(0) &+ ip_2(0) = 0 + i\lambda r,
\end{aligned}
\end{equation}
for some $b>0$ and $r\geq 0,$ so that $r^2 + 2b^2 < \delta^2.$ The case $b<0$ is analogous. Such trajectories escape $B_{\lambda \delta}(0)$ both forward and backward in time, so they are called escaping trajectories.

Consider the symplectic frame $\{Y_1,Y_2\}$  induced by the quaternions in coordinates $(q,p)$. Notice that it may not coincide with the initial frame $\{X_1, X_2\}$. 
Let $\eta(t)$ be a continuous argument of a non-trivial linearized solution along $\gamma$ in the frame $\{Y_1, Y_2\}$. 
Observe that 
\begin{equation}\label{cosh}
q_1(t)^2 +p_1(t)^2=2\lambda ^2 b^2 \cosh (2\bar \alpha t)   <\lambda^2\delta^2
\end{equation}
for every $t$ satisfying $\gamma(t) \in B_{\lambda \delta}(0)$.  Let $T>0$ be such that $\gamma(\pm T) \in \partial B_{\lambda \delta}(0).$

\vspace{0.3cm}
\noindent
{\bf Claim.}  If $\delta>0$ is sufficiently small, then
$
\eta(T) - \eta(-T) > \omega T - 2\pi.
$
\vspace{0.2cm}

We follow the computation in \cite{dePaulo_Salomao} to prove the claim. Since
 \begin{align*}  
\kappa_{11} &= \frac{1}{\lvert \nabla K \rvert^2} \left( \bar \alpha ^2 \bar \omega (q_1^2 + p_1^2) - 2 \bar \alpha \bar \omega^2 q_2p_2 + \bar r(q_1 q_2+ p_1 p_2 )^2    \right),\\
\kappa_{12} &=   \frac{1}{\lvert \nabla K \rvert^2} \left( \bar \alpha   \bar \omega^2 (p_2^2 - q_2^2) + \bar r(  q_2 p_1 - q_1 p_2)(q_1q_2 + p_1 p_2)    \right)             ,\\ 
\kappa_{22} &=    \frac{1}{\lvert \nabla K \rvert^2} \left(  \bar \alpha ^2 \bar \omega  (q_1^2 +p_1^2) +  2  \bar   \alpha \bar \omega^2 q_2p_2 + \bar r(q_2 p_1 - q_1p_2)^2    \right)              ,\\
\kappa_{33} &=    \frac{1}{\lvert \nabla K \rvert^2} \left( 2 \bar \alpha^3 q_1p_1 + \bar \omega^3 (q_2^2 + p_2^2)     \right)   ,
\end{align*}
where the functions $\bar r:=r_{11}\bar \omega ^2 +2r_{12}\bar
\alpha \bar \omega +r_{22}\bar \alpha^2$, $r_{ij}:=\partial^2_{I_iI_j}
R 
, \; i,j=1,2,$ are constant along the trajectories,
we use \eqref{eq:varangle}, \eqref{solutions}, and \eqref{signb}
to find
$$
\begin{aligned}
\dot \eta =  & \bar \omega  +\frac{1}{2\bar \alpha^2 b^2\cosh (2\bar \alpha
t) + \bar \omega^2 r^2}\{\epsilon 2\bar \alpha^3 b^2+( \bar \alpha \bar
\omega^2r^2 -\epsilon \lambda^2b^2r^2 \bar r)\sin (2\eta(t)-2\bar \omega t) +\\
& \lambda^2b^2r^2\bar r[\cosh (2\bar \alpha t) + \sinh (2 \bar \alpha t) \cos (2\eta(t)-2\bar \omega t)]\}.
\end{aligned}
$$
Notice that $\bar r$ is uniformly bounded on $B_\delta(0)$ by some constant $c_1>0$ that does not depend on $\delta$. The sign $\epsilon\in \{+1,-1\}$ in the expression for $\dot \eta$ depends on the initial conditions in \eqref{signb}. Since $\bar \alpha>0$ for $\delta>0$ small, we may assume that $\epsilon = -1.$ The case $\epsilon=+1$ is simpler due to the positivity of the term $\epsilon2\bar \alpha^3b^2$ and was treated in \cite{dePaulo_Salomao}. First we consider the case $\bar \omega^2 r^2 \leq 4\bar \alpha^2 b^2.$
Using \eqref{cosh}, we obtain for every $\delta>0$  sufficiently small
$$
\begin{aligned}
\dot \eta - \bar \omega & >\frac{-\bar \alpha (2\bar \alpha^2b^2 + \bar \omega^2 r^2) -2c_1r^2\delta^2} {2\bar \alpha^2 b^2 \cosh(2\bar \alpha t) +\bar \omega^2 r^2} \\
& \geq  \frac{-\bar \alpha(2\bar \alpha^2b^2 +  4\bar \alpha^2 b^2 )-2c_1\delta^2 4\bar \alpha^2 b^2/\bar \omega^2}{2\bar\alpha^2 b^2\cosh (2\bar \alpha t)+\bar \omega^2 r^2}\\
& > \frac{-4\bar \alpha}{\cosh (2\bar \alpha t)}.
\end{aligned}
$$
Since 
$$
0<\int_{-T}^{T} \frac{\bar \alpha}{\cosh(2\bar \alpha t)} \, dt < \int_{-\infty}^{+\infty} \frac{\bar \alpha}{\cosh(2\bar \alpha t)} \, dt = \frac{\pi}{2},
$$
we obtain 
$
\eta(T) - \eta(-T) > \bar \omega 2 T -  2\pi> \omega T - 2\pi,
$
proving the claim in the case $\bar \omega^2 r^2 \leq 4\bar \alpha^2 b^2$.

 Now assume that 
$\bar \omega^2 r^2  > 4\bar \alpha^2 b^2.$
If $\delta>0$ is sufficiently small and $\eta(t_*) = \bar \omega t_* + \frac{\pi}{4}+ k_0 \pi$ for some $k_0\in \Z,$ then 
$$   
\begin{aligned}
    \dot \eta(t_*) & = \bar \omega +\frac{ \bar \alpha (\bar
\omega^2r^2- 2\bar \alpha^2 b^2)+\lambda^2 b^2r^2 \bar r  +
\lambda^2 b^2r^2\bar r\cosh (2\bar \alpha t_*)}{2\bar \alpha^2 b^2\cosh (2\bar \alpha
t_*) + \bar \omega^2 r^2}\\
& = \bar \omega + \frac{ \bar \alpha( \bar \omega^2 r^2 - 2\bar \alpha^2b^2) + r^2O(\delta^2)}{2\bar \alpha^2 b^2\cosh (2\bar \alpha
t_*) + \bar \omega^2 r^2}\\
& > \bar \omega + \frac{ \bar \alpha\bar \omega^2r^2/2 +r^2O(\delta^2)}{2\bar \alpha^2 b^2\cosh (2\bar \alpha
t_*) + \bar \omega^2 r^2}\\
& > \bar \omega.
\end{aligned}
$$
This forces $\eta(T)- \eta(-T) > \bar \omega 2T - \pi > \omega T - \pi,$
and the claim is proved.

We see from \eqref{solutions} that given $M_0>0$, there exists $0<\delta_{M_0}\ll \delta$ so that if $\gamma([-T,T])\cap B_{\lambda \delta_{M_0}}(0)\neq \emptyset$, then $T> (M_0+2\pi)/\omega$ and thus in view of the claim above we obtain $\eta(T) - \eta(-T) > M_0$.

Now consider the frame $\{X_1, X_2\}$ as in the statement, defined in local coordinates  $(x,y)$.  Up to a projection along the Hamiltonian vector field, we may assume that $\{X_1, X_2\}$ and $\{Y_1, Y_2\}$ span the same plane field transverse to the flow. Hence we may write $Y_1 = a  X_1 + b X_2,$ where $a,b\colon\text{closure}(B_\delta(0)) \setminus \{0\} \to \R$ are smooth. Since $\text{closure}(B_\delta(0)) \setminus \{0\}$ is simply connected, $Y_1\equiv a+ib$ admits a well-defined continuous argument, that is $a+ib \in \R^+ e^{i\hat \zeta}$ for some smooth function $\hat \zeta\colon\text{closure}(B_\delta(0)) \setminus \{0\} \to \R$. 
Let  $C_\lambda:=\sup_{z_1,z_2 \in \partial B_{\lambda \delta}(0)} |\hat \zeta(z_1)-\hat \zeta(z_2)|<+\infty.$ Since the frames $\{X_1,X_2\}$ and $\{Y_1,Y_2\}$ have a $C^\infty_{\text{loc}}$-limit under the re-scaling $K(\lambda q,\lambda p)/\lambda^2\to -\alpha I_1 + \omega I_2$ as $\lambda \to 0^+$,  we may assume that $\sup_{0<\lambda \leq 1} C_\lambda < C<+\infty.$ Let $\gamma(t), t\in [-T,T],$ be a escaping trajectory as above so that $\gamma(0)\in B_{\lambda \delta_{M_0}}(0)$ and $\gamma(\pm T)\in \partial B_{\lambda \delta}(0)$. Consider a non-trivial solution to the linearized flow along $\gamma(t)$ in the frame $\{X_1, X_2\}$, let $\theta(t),\eta(t)$ be continuous arguments of this solution in the frames $\{X_1,X_2\}$ and $\{Y_1,Y_2\}$, respectively. Then $\theta(T) - \theta(-T) > \eta(T) - \eta(-T) - 2C-2\pi.$ Hence, choosing $M_0>0$ as above so that $M_0 > 2C+\pi + M,$ we obtain that $\theta(T) - \theta(-T)> M,$ as desired. 

Finally, taking $\delta_M>0$ even smaller if necessary, we obtain the desired sets in local coordinates by defining $U_1:=\text{closure}(B_\delta(0))$ and $U_M := \text{closure}(B_{\delta_M}(0))$.
\end{proof}

Now, we use Lemma \ref{lem_deltatheta} and some properties of mechanical systems to show that every periodic orbit sufficiently close to a saddle-center has an arbitrarily high index.

\begin{proposition}\label{prop_indicealto} Let $0\in \R^4$ be a saddle-center of a real-analytic mechanical Hamiltonian $H(x,y)=|y|^2/2 + V(x)$ and let $M>0$ be given. Then  there exists a small compact neighborhood $V_M\subset \R^4$ of $0$ such that if $\gamma\subset H^{-1}(E), \gamma \neq P_{2,E},$ is a periodic orbit  intersecting $V_M$, then $\mu(\gamma)>M.$ Here, $P_{2,E}\subset H^{-1}(E)$ is the Lyapunov orbit near $0$. 
\end{proposition}

\begin{proof} 
Consider the transverse frame $\{X_1, X_2\}$ induced by the quaternions and defined on the regular points of $H$.  Consider also the transverse frame on $TH^{-1}(E) \setminus \{y_1=y_2=0\}$ induced by the vertical and horizontal lifts of $(y_1,y_2)^\perp = (-y_2,y_1)$ under the projection $(x,y)\mapsto x$. Indeed,  there exist non-vanishing independent vector fields $\hat X_1,\hat X_2$ tangent to $H^{-1}(E) \setminus \{y_1=y_2=0\}$ given by
$$
\begin{aligned}
 \hat X_1 & = (0,0,-y_2,y_1)\\
& = - (y_1V_{x_1}+y_2V_{x_2})gX_1 + 2(E-V(x))gX_2 + (y_2V_{x_1}-y_1V_{x_2})X_3,\\
 \hat X_2 & = (2(E-V(x)))^{-1}(-y_2,y_1,-V_{x_2},V_{x_1})\\
& =-(2(E-V(x))g)^{-1}X_1,\\
\end{aligned}
$$
where $V_{x_i} = \partial_{x_i}V$, $i=1,2,$ $g =(V_{x_1}^{2}+V_{x_2}^{2}+y_{1}^{2}+y_{2}^{2})^{-1/2},$  so that  ${\rm span}\{\hat X_1,\hat X_2\}$ is transverse to the Hamiltonian vector field $X_H=(y_1,y_2,-V_{x_1},-V_{x_2})$. A linearized solution $\eta = \alpha_1 \hat X_1 + \alpha_2 \hat X_2 +\alpha_3X_H,$ with $\alpha_1^2+\alpha_2^2 \neq 0,$ along a trajectory $\gamma$ satisfies the following condition: if
$\alpha_2 =0$ and $\alpha_1>0$, then  $\dot \alpha_2 >0.$
Indeed, from $\dot \eta = DX_H(\gamma) \eta$ and $DX_H \cdot X_H = \frac{d}{dt}X_H \circ \gamma$, we obtain
$$
\alpha_1 \langle DX_H\cdot \hat X_1, \hat X_2\rangle = \alpha_1 \langle\hat X_1',\hat X_2\rangle + \dot \alpha_2 \langle\hat X_2,\hat X_2\rangle,
$$
where $\hat X_1' = \frac{d}{dt}\hat X_1 \circ \gamma$. Since  
$$
\langle DX_H \cdot \hat X_1, \hat X_2\rangle = \langle(-y_2,y_1,0,0),\hat X_2 \rangle > 0,
$$
and
$$
\langle\hat X_1',\hat X_2\rangle=\langle (0,0,V_{x_2},-V_{x_1}),\hat X_2\rangle \leq 0,
$$
we conclude that $\dot \alpha_2>0$ if $\alpha_2=0$ and $\alpha_1>0$. In particular, the linearized flow in the frame $\{\hat X_1,\hat X_2\}$ does not rotate backward more than $\pi$. This means that if $\eta(t)$ is a continuous argument of $\alpha_1(t) + i \alpha_2(t)\neq 0$ and $\eta(t) = \eta_*,$ then $\eta(s) > \eta_* - \pi$ for every $s>t$.  Since $\hat X_2$ is parallel to $X_1$, the frame induced by $\{\hat X_1,\hat X_2\}$ and $\{X_1,X_2\}$ do not wind with respect to each other. Hence, the argument $\theta(t)$ of a linearized solution in the frame $\{X_1, X_2\}$ satisfies a similar property.  As a final remark, if a simple periodic orbit touches the boundary of the Hill region, then the contribution to the variation of the argument $\theta(t)$ is bounded since that happens at most twice along the minimal period. 

Given $M'>0$, we know from Lemma \ref{lem_deltatheta} that there exist small neighborhoods $U_{M'}\subset U_1\subset \R^4$ of $0$ so that if $\gamma\neq P_{2,E}\subset H^{-1}(E), E>0$ small, is a simple periodic trajectory intersecting $U_{M'}$, and not contained in $U_1$, then the variation in the argument $\eta(t)$ of a non-trivial transverse linearized solution inside $U_1$ in the frame $\{X_1,X_2\}$ is greater than $M'$. From the reasoning above, the total variation of $\eta(t)$ along the whole period of $\gamma$  is greater than $ M' - C$, for some fixed $C>0$ that does not depend on $\gamma$ or $M'$. If $M'>0$ is taken sufficiently large, then $\mu(\gamma) > M$, as desired. \end{proof}

\begin{corollary} Let $0\in \R^4$ be a saddle-center of a real-analytic mechanical Hamiltonian $H(x,y) =|y|^2/2 + V(x)$.  Given $M>0$, there exists $E_M>0$ so that if $0<E<E_M$ and $\gamma\subset H^{-1}(E)$ is a simple periodic orbit, linked with the Lyapunov orbit $P_{2, E}\subset H^{-1}(E)$ near $0$, then $\mu(\gamma) > M.$ Iterates of $\gamma$ also have index greater than $M$.
\end{corollary}

\begin{proof}
    Take $V_M$ as in Proposition \ref{prop_indicealto}. Consider the local coordinates $(q,p)$ as in the proof of Lemma \ref{lem_deltatheta}, so that the Hamiltonian has the form $K=-\alpha I_1+ \omega I_2 + R(I_1,I_2)$, where $I_1=q_1p_1$ and $I_2 = (q_2^2+p_2^2)/2$. Recall that in these coordinates, the Lyapunov orbit is given by $P_{2,E} = \{q_1=p_1 =0\} \cap K^{-1}(E),$ where $E>0$ is small. Notice that $P_{2,E}$ is the boundary of the embedded disks
$$
\begin{aligned}
U_{1,E} & := \{ p_1=-q_1\geq 0, I_1(E)\leq q_1p_1 \leq 0, I_2 \leq I_2(E)\} \cap K^{-1}(E),\\
U_{2,E} & := \{ p_1=-q_1\leq 0, I_1(E)\leq q_1p_1 \leq 0, I_2 \leq I_2(E)\} \cap K^{-1}(E),
\end{aligned}
$$
where $I_1(E)=-\frac{E}{\alpha} + O(E^2)<0$ solves $E = -\alpha I_1 +R(I_1,0)$, and $I_2(E)>0$ solves $E= \omega I_2 +R(0,I_2)$, for every $E>0$ small.  Observe that the interior of such disks is transverse to the flow, and $I_1(E), I_2(E) \to 0$ as $E \to 0$. Hence we find $E_M>0$ sufficiently small such that $U_{1,E},U_{2,E}\subset V_M$ for every $0<E<E_M$. In particular,
 any periodic orbit $\gamma\subset H^{-1}(E),$ with $0<E<E_M,$ that is linked with $P_{2,E}$, must intersect $U_{1,E}$ and $U_{2,E}$ and thus intersects $V_M$. This implies $\mu(\gamma)>M$.  
\end{proof}

\begin{proof}[Proof of Proposition \ref{prop_weakconv}]
Recall from the proof of Lemma \ref{lem:sl} that in re-scaled coordinates $\hat x=x/\sqrt{\epsilon}$ near the saddle $v_i$, the re-scaled potential $\tilde V_\epsilon(\hat x) $ converges in $C^\infty_{\rm loc}$ to $\tilde V_0(\hat x): = a \hat x_1^2/2 + b \hat x_2^2/2 + f(\hat x_1)$ as $\epsilon \to 0^+$, where $a<0$, $b>0$, and $f$ satisfies \eqref{eq_cutfunction}.  Denote $\tilde H_\epsilon(\hat x,\hat y):=|\hat y|^2/2+\tilde V_\epsilon(\hat x).$ Recall that $\{\tilde V_\epsilon\leq 0\}$ contains a disk-like region $\tilde K_{i,\epsilon} \subset \{0\leq \hat x_1<2\}$ with a singularity at $0\in \R^2$, corresponding to the capping of $K_0$ near  $v_i$.

For $M>0$ large, consider the sets $U_M \subset U_1$ as in Lemma \ref{lem_deltatheta}. In re-scaled coordinates $(\hat x,\hat y)$, they become $\tilde U_M:=\epsilon^{-1}U_M \subset \tilde U_1:=\epsilon^{-1} U_1$. There exists $\delta_1>0$ such that for every $\epsilon >0$ sufficiently small, we find $0<\lambda\leq 1$ satisfying $B_{\delta_1}(0) \subset \lambda \tilde U_M$ and $\lambda \tilde U_1 \subset \{\hat x_1 <1\}$. This implies that for $M>0$   large,   any periodic orbit  $P \subset  \tilde H^{-1}_\epsilon(E), P\neq P_{2,E},$ intersecting $B_{\delta_1}(0)$, has index greater than~$3$. 

We shall find  $\epsilon>0$ small enough so that $\tilde H_\epsilon^{-1}(E)$ is dynamically convex on $0<\hat x_1< 2$ for every $E>0$ sufficiently small. Indeed, assume by contradiction the existence, for every $\epsilon>0$ small and $E>0$ arbitrarily close to $0$, of an orbit $\tilde Q_{2,\epsilon, E}$ contained in the capped region $\tilde H_\epsilon^{-1}(E)\cap \{0<\hat x_1< 2\}$ so that its index is $\leq 2$. From the considerations above, we find $\epsilon_1>0$ small so that $\tilde Q_{2,\epsilon,E} \subset \{\epsilon_1<\hat x_1<2\}$ for every $\epsilon>0$ sufficiently small and $E>0$ arbitrarily small. The limiting dynamics as $\epsilon\to 0^+$ is determined by the decoupled Hamiltonian $\tilde H_0=\tilde H_1(\hat x_1,\hat y_1) + \tilde H_2(\hat x_2,\hat y_2)$, where  $\tilde H_1=\hat y_1^2/2 +a \hat x_1^2/2 + f(\hat x_1)$ and $\tilde H_2=\hat y_2^2/2 + b \hat x_2^2/2$.  
Notice that the orbit $\tilde Q_{2,\epsilon, E}\subset \tilde H^{-1}_\epsilon(E), E>0$ small, cannot get arbitrarily close to $0$ in the $(\hat x_2,\hat y_2)$-plane as $E \to 0^+$, otherwise it would intersect $\hat x_1 = \epsilon_1$, a contradiction. Hence, if $\epsilon>0$ is fixed sufficiently small, the action of $\tilde Q_{2,\epsilon, E}$ is uniformly bounded. Otherwise, its index would be $>2$ due to the contribution to the index of the linearized flow on the $(\hat x_2,\hat y_2)$-plane. Taking the limit $E \to 0^+,$ and then $\epsilon\to 0^+$, we obtain a periodic orbit $\tilde Q\subset \tilde H_0^{-1}(0)\cap \{\epsilon_1 \leq \hat x_1 <2\}$ with index $\leq 2$. This is a contradiction with Proposition \ref{prop_decoupled1} below; in fact, it follows from \eqref{eq_cutfunction} and Proposition \ref{prop_decoupled1} that, except for the index-2 orbit $\hat P_{2,0,E}=\{\hat x_1=\hat y_1=0, b\hat x_2^2+\hat y_2^2=2E\}$, all periodic orbits in $\tilde H_0^{-1}(E)$ have index $\geq 3$.  We conclude that $\tilde H_\epsilon^{-1}(E)$ is dynamically convex in the capped region $0<\hat x_1<2$ for every  $\epsilon>0$ fixed sufficiently small and $E>0$ sufficiently small. 

For the original capped Hamiltonian $H_\epsilon$, with $\epsilon>0$ sufficiently small as above, we see that since $S_0$ is dynamically convex and admits a positive frame, $W_{\epsilon, E}$ is weakly convex and the Lyapunov orbits near $p_i$ are the only index-$2$ orbit for $E>0$ sufficiently small. Indeed, if another orbit $Q_{2,\epsilon, E}\subset W_{\epsilon, E}$ exists, then the argument above shows that it cannot be contained in the capped region and stays away from all $p_i$ as $E \to 0^+$. Due to the positive frame for $S_0$, its period is uniformly bounded in $E>0$, and thus converges up to a subsequence to a periodic orbit $Q_{2,\epsilon,0}\subset S_0 \setminus \{p_1,\ldots,p_l\}$. Its index is $\leq 2$,  contradicting the dynamical convexity of $S_0$.

Finally, let us check that the Lyapunov orbits have the smallest actions among all periodic orbits. We shall see in Section \ref{sec:liouville} below that $H^{-1}_\epsilon(E)$ admits a  contact form $\lambda_E$ which converges in $C^\infty_{\text{loc}}$ to $\lambda_0$ as $E \to 0^+$, where $\lambda_0$ is a contact form on $H_\epsilon^{-1}(0)\setminus \{p_1,\ldots,p_l\}$. 
The action of a periodic orbit $Q$ of $\lambda_E$ is defined by $\int_{Q} \lambda_E$. Fix $\epsilon>0$ sufficiently small as before and consider re-scaled coordinates $(\tilde x,\tilde y) = (x/\sqrt{E},y/\sqrt{E})$ near some $p_i$. Let  $\tilde H_E (\tilde x, \tilde y) := H_\epsilon(x,y)/E$. Then $H_\epsilon^{-1}(E)$ corresponds to $\tilde H_E^{-1}(1)$. As $E \to 0$, the Lyapunov orbit $P_{2,E}\subset H_\epsilon^{-1}(E)$ in coordinates $(\tilde x,\tilde y)$ converges in $C^\infty$  to $\{\tilde x_1=\tilde y_1=0, b\tilde x_2^2+\tilde y_2^2=2\}$. Hence, for $E>0$ sufficiently small, $P_{2,E}$ has action $2\pi E/ \sqrt{b} + O(E^2),$ that is $\mathcal{A}(P_{2,E}) = O(E)\to 0$ as $E \to 0$. Suppose, by contradiction, that, for every $E>0$, there exists a periodic orbit $Q_E\subset  H_\epsilon^{-1}(E)$ so that $Q_E$ is not a Lyapunov orbit and $\mathcal{A}(Q_E) \to 0$ as $E \to 0$. Since the Lyapunov orbits are hyperbolic and $Q_E$ is not a Lyapunov orbit, we see that $Q_E$ cannot converge to any $p_i$. Hence, up to a subsequence, $Q_E$ converges to non-constant trajectories of $\lambda_0$, possibly homoclinic/heteroclinic orbits connecting the saddle-centers,  contradicting  $\mathcal{A}(Q_E)\to 0$ as $E \to 0^+$.  
\end{proof}

\subsection{The contact property}\label{sec:liouville}
In this section, $H=H_\epsilon =\frac{1}{2}|y|^2+V_\epsilon(x)$ is as in Proposition \ref{prop_weakconv}, where $\epsilon>0$ is fixed sufficiently small. We show that for $E>0$ sufficiently small, the regular sphere-like hypersurface $W_E= H^{-1}(E)$ near $S_0$  has contact type, and the induced contact structure $\xi=\ker \lambda_E$ is tight; see, for instance, \cite[Lemma 4.1]{abbas2019holomorphic}. The critical set $H^{-1}(0)$ is denoted by $W_0$. It contains $S_0$ whose projection to the $x$-plane is the compact disk-like set $K_0$.

In the next section, we shall define an almost complex structure $J_E$ on $\R \times W_E$, which admits a pair of holomorphic planes asymptotic to the Lyapunov orbits, see section \ref{sec:holocurves}. To control the location of these holomorphic planes, it will be necessary to choose a particular contact form $\lambda_E$ on $W_E$ as in the following statement.

\begin{proposition}\label{prop:transverseLiouville} The following assertions hold.
\begin{itemize}
    \item[(i)] For every $E>0$ sufficiently small,  there exists a Liouville vector field  $X_E$ defined on a neighborhood of the sphere-like hypersurface $W_E=H^{-1}(E)\subset \R^4$, which is  transverse to $W_E$ and the induced contact form $\lambda_{E}:= -\iota_{X_E}\omega_0$ on $W_{E}$ is tight. Moreover, for suitable symplectic coordinates near each saddle-center $p_i$ so that the Hamiltonian has the form $H = \frac{1}{2}(y_1^2 +y_2^2+  a x_1^2  + b x_2^2) + R(x), $ where $a<0, b >0, R = O(|x|^3),$ we can assume that $X_E = \frac{1}{2}(x_1\partial_{x_1} + x_2 \partial_{x_2} + y_1 \partial_{y_1} + y_2 \partial_{y_2})$ and thus $\lambda_E=\frac{1}{2} (x_1dy_1-y_1dx_1+x_2dy_2-y_2dx_2)|_{W_E}$  for every $|x_1|<c\sqrt{E}$, where $c>0$ is independent of $E>0$. 
    \item[(ii)] There exists a Liouville vector field $X_0$ defined on a neighborhood of $\dot W_0:= H^{-1}(0) \setminus \{p_1,\ldots,p_l\}$, which is transverse to $\dot W_0$. Given any neighborhood $\mathcal{U}\subset \R^4$ of $\{p_1,\ldots,p_l\}$, we have $X_0=X_E$ outside $\mathcal{U}$ for every $E>0$ sufficiently small. In particular,  $\lambda_E \to \lambda_0$ in $C^\infty_{\text{loc}}(\dot W_0)$ as $E\to 0^+$, where $\lambda_0:=-\iota_{X_0}\omega$ is the induced contact form on $\dot W_0$. 
    \end{itemize}
\end{proposition}

\begin{proof} We need to interpolate a few Liouville vector fields to obtain the desired $X_E$.  
  We start with the Liouville vector field $$\mathcal{Z}_{0}:=y_{1}\partial_{y_{1}}+y_{2}\partial_{y_{2}},$$ which is transverse to $\dot{W}_{0} $ except at the points of $\dot W_0$ projecting to $\partial \Omega_0$. The set $\dot{W}_{0}\cap\{y=0\}\subset H^{-1}(0)$ is formed by finitely many embedded curves $\gamma_{1}, \dots, \gamma_{n}, n=2l,$ projecting to the corresponding embedded curves $\hat \gamma_1,\ldots, \hat \gamma_n\subset \partial\Omega_{0}\setminus  \{v_1,\ldots,v_l\}.$

Fix $\gamma=\gamma_{j}$ for some $j$, and consider the quaternion vector fields  $X_0,X_1,X_2,X_3$  defined on regular points of $H$, see \eqref{eq_quaternionic_mechanical}.     
 Since   $T {\gamma}=\R X_{2}|_\gamma$, the plane field $\mbox{span}\{X_{1}, X_{3}\}|_{\gamma}\subset T\dot W_0$  is transverse to   $\gamma$.
Take a parametrization $\gamma = \gamma(t), t\in I=[0,1]$, so that $\gamma(0) = p_i =0$ for some $i$, and consider  coordinates $(\tilde x, \tilde y, \tilde z, t)$ on a tubular neighborhood  $\mathcal{U} \equiv \tilde B:=  B_{r}(0)\times I$,   $r>0$ small,   of $\gamma$, given by $(\tilde{x}, \tilde{y}, \tilde{z}, t)\mapsto \exp_{\gamma(t)}(\tilde{x}X_{1}+\tilde{y}X_{3}+\tilde{z}X_{0})\in \R^4.$
Here, $B_{r}(0)\subset\R^{3}$ is the open ball of radius $r>0$ centered at $0$ and  $\exp$ is the exponential map induced by the Euclidean metric.
  Let $\beta=\beta(\tilde x ,\tilde  y,\tilde z,t)\in[0,1]$ be a smooth  function   so that $\beta =0$ near $\partial \tilde B$ and $\beta=1$ near $L:=0 \times (r,1-r)$.
 Let $f(\tilde{x}, \tilde{y}, \tilde{z}, t):=- \tilde y \beta(\tilde{x}, \tilde{y}, \tilde{z}, t)$ be defined on $\tilde B$.
 Then $X_f$ vanishes near $\partial \tilde B$ and $X_{f}=X_{0}$ is positively transverse to $\dot{W}_{0}$ near $L$. Hence  $$\mathcal{Z}_{\varepsilon}:=X_{\varepsilon f}+\mathcal{Z}_{0}$$
is transverse to  $\dot{W}_{0}$ near $L$ for every $\varepsilon>0$ sufficiently small. Repeating the same construction near each $\gamma_j$, we obtain a Liouville vector field still denoted by $\mathcal{Z}_\varepsilon$, which is positively transverse to $\dot W_0$ away from arbitrarily small neighborhoods of $p_i,i=1,\ldots,l$.  We may assume the existence of functions $f_i=f$ as above, defined on mutually disjoint open neighborhoods $\mathcal{U}_i=\mathcal{U}$ of $\gamma_i$ and they vanish near  $\partial \mathcal{U}_i.$ They give rise to a smooth function $f$ supported in the interior of $\cup_i \mathcal{U}_i.$

Let $p_c:=p_i$ for some $i$. We may assume that  $p_c=0\in \R^4$ and near $p_c$ the Hamiltonian writes as
$H(x,y)=y_{1}^{2}/2+y_{2}^{2}/2+a x_{1}^{2}/2+b x_{2}^{2}/2+R(x)
$, where $a<0,b>0$ and $R=O(|x|^3)$. We consider the subset of $\dot W_0$ near $0$ that is contained in $\{x_1>0\}$. Then there exists $C>0$ such that 
\begin{equation}\label{x1x2}
    x_1 \geq C|x_2|\geq 0
\end{equation}
for every $(x_1,x_2,y_1,y_2) \in \dot W_0\cap \{x_1 > 0\}\cap B_{\delta_1}(0)$, where $\delta_1$ is fixed sufficiently small.
Let 
$$
\mathcal{Z}:=\frac{1}{2}(x_1-\delta_1) \partial_{x_1} +\frac{1}{2} x_2 \partial_{x_2} + \frac{1}{2}y_1 \partial_{y_1} +\frac{1}{2} y_2 \partial_{y_2}
$$ be the radial Liouville vector field centered at $ (\delta_1,0,0,0).$
Then $\mathcal{Z}$ is positively transverse to $\dot{W}_0$ on $\{(x_1 - \delta_1)^2 + x_2^2 <(3\delta_1) ^2, x_1>0\}$, provided $\delta_1>0$ is small enough. Indeed, we find 
\[ 
  d H \cdot \mathcal{Z}|_{\dot W_0} = - \frac{1}{2}\delta_{1} a  x_1  -R + \frac{1}{2}( x_1 - \delta_1)\partial_{x_1}R  + \frac{1}{2}x_2 \partial_{x_2}R, 
\] 
which is positive in that region if $\delta_1>0$ is sufficiently small. We have used \eqref{x1x2}. Furthermore, this also shows that for every $E>0$ sufficiently small, we have
\begin{equation}\label{dHZ}
d H \cdot \mathcal{Z} =E- \frac{1}{2}\delta_1ax_1 -R   +\frac{1}{2}(x_1-\delta_1)\partial_{x_1}R+\frac{1}{2}x_2\partial_{x_2}R> 0
\end{equation} on $W_E\cap \{(x_1-\delta_1)^2 + x_2^2<(3 \delta_1) ^2,x_1 \geq 0\},$ 
if $\delta_1>0$ is fixed sufficiently small, see \eqref{x1x2}.
On the other hand, from the construction above, we may assume that $\mathcal{Z}_\epsilon$ is positively transverse to $\dot W_0$ on  $\{(2\delta_1)^2 \leq (x_1-\delta_1)^2 + x_2^2\leq (4\delta_1)^2,x_1> 0\}.$  This can be achieved by taking $r>0$ sufficiently small.

Now we construct an interpolation between $\mathcal{Z}$ and $\mathcal{Z}_\varepsilon$ on  $\{(2\delta_1)^2 \leq (x_1-\delta_1)^2 + x_2^2 \leq (3\delta_1)^2,x_1> 0\}$. Let
$
\ell(x_1,x_2,y_1,y_2)  := -(x_1-\delta_1) y_1/2 - x_2 y_2/2+\varepsilon f,
$
where $f$ is as above. 
Notice that the Hamiltonian vector field of $\ell$ satisfies 
$
X_{\ell} = \mathcal{Z}_0 + X_{\varepsilon f} - \mathcal{Z} = \mathcal{Z}_\varepsilon - \mathcal{Z}.
$ 
Let  $\tilde h:= h((x_{1}-\delta_1)^2+x_2^2)
$
be defined near $0\in \R^4$, where $h\colon\R \to [0,1]$ is  smooth and satisfies
$h(t)=0$ if $t\leq(2\delta_1)^2$, $h(t)=1$ if  $t\geq (3\delta_1)^2,$ and $h'(t)>0$ if $(2\delta_1)^2< t< (3\delta_1)^2$.
Let 
$$
\tilde{\mathcal{Z}}_\varepsilon:= \mathcal{Z}+ X_{\ell\tilde h}.
$$
Then $\tilde{\mathcal{Z}}_\varepsilon = \mathcal{Z}$ on $\{(x_1-\delta_1)^2 + x_2^2 \leq (2\delta_1)^2, x_1>0\}$ and $\tilde{\mathcal{Z}}_\varepsilon = \mathcal{Z}_\varepsilon$ on $\{(x_1-\delta_1)^2 + x_2^2 \geq (3\delta_1)^2, x_1 >0\}$. 


We claim that $\tilde {\mathcal{Z}}_\varepsilon$ is transverse to $\dot{W}_0$ near $0$ for every $\varepsilon>0$ sufficiently small.
It suffices to consider the set $A_0:=\{(2\delta_1)^2<(x_1-\delta_1)^2+x_2^2<(3\delta_1)^2, x_1>0\}\cap \dot W_0.$
Notice that
\begin{align*}
d H \cdot \tilde{\mathcal{Z}}_\varepsilon 
 = (1- \tilde h) d H \cdot \mathcal{Z} + \tilde h d H \cdot \mathcal{Z}_\varepsilon + \ell d H \cdot X_{\tilde h}.
\end{align*}
The set $\{(x_1-\delta_1)^2+x_2^2=(3\delta_1)^2,x_1>0\}\cap \{y=0\}\cap \dot W_0$ is formed by two points $p_1,p_2$ projecting to the boundary of the Hill 
region. Since $dH\cdot \tilde{ \mathcal{Z}}_\varepsilon|_{p_i} = dH \cdot \mathcal{Z}_\varepsilon|_{p_i}=\epsilon dH \cdot X_f,i=1,2,$ we find a neighborhood $\mathcal{U}\subset \dot W_0$ of $\{p_1,p_2\}$ and $\hat c>0$ such that $dH \cdot \tilde{\mathcal{Z}}_\varepsilon|_{\mathcal{U}} > \hat c \epsilon>0$ for every $\epsilon>0$ sufficiently small. Notice that $\hat c>0$ does not depend on $\epsilon.$
Since both $\mathcal{Z}$ and $\mathcal{Z}_\varepsilon$ are positively transverse to $A_0$, we find  $\sigma>0$, independent of $\epsilon,$ such that $(1- \tilde h) d H \cdot \mathcal{Z} + \tilde h d H \cdot \mathcal{Z}_\varepsilon>\sigma$ on $A_0\setminus \mathcal{U}$.    
For the last term of $dH\cdot \tilde{\mathcal{Z}}_\varepsilon$, we compute
$$
\begin{aligned}
\ell d H \cdot X_{\tilde h} & =h'\cdot [((x_1-\delta_1)y_{1}+x_2y_2)^2
-2 \varepsilon f \cdot ((x_1-\delta_1)y_1+x_2y_2).
\end{aligned}
$$
Recall that $f$ is bounded, and $\varepsilon>0$ can be taken arbitrarily small. Since $h'\geq 0$, we conclude that $\ell d H \cdot X_{\tilde h}|_{A_0 \setminus \mathcal{U}} >-\sigma$ for every $\varepsilon>0$ sufficiently small.
Hence $d H \cdot \tilde{\mathcal{Z}}_\varepsilon>0$ on 
$A_0$ for every $\varepsilon>0$ sufficiently small, from which the claim follows. This fact also implies that for fixed $\varepsilon>0$ sufficiently small, $\tilde{\mathcal{Z}}_\varepsilon$ is  positively transverse to $W_E\cap \{(x_1-\delta_1)^2+x_2^2< (3\delta_1)^2,x_1\geq 0\}$, for every $E>0$ sufficiently small, see \eqref{dHZ}.

The last step is to construct an interpolation between $\tilde{\mathcal{Z}}_\varepsilon$ and  
\begin{equation}\label{Liouville_vf}
Y:=\frac{1}{2} ( x_1 \partial_{x_1} + x_2 \partial _{x_2} +y_1 \partial_{y_1} +y_2 \partial_{y_2} )
\end{equation} 
near $p_c=0$. This interpolation depends on the energy $E>0$. Recall that $\tilde{\mathcal{Z}}_\varepsilon=\mathcal{Z}$ on $(x_1-\delta_1)^2 + x_2^2 \leq (2\delta_1)^2$. 
We fix $E>0$ sufficiently small and construct a Liouville vector field $\mathcal{Z}_E$ transverse to $W_E$, which coincides with $Y$ for $|x_1|$ sufficiently small. We restrict to the region $\{x_1^2+x_2^2<\delta_1^2\},$ where $\tilde{\mathcal{Z}}_\varepsilon$ coincides with $\mathcal{Z}.$

Let $0<\delta_E\ll\delta_1$ be a small number to be determined below,  and let $\tilde f_{E}=\tilde f_E(x_1)$ be a smooth function satisfying 
$\tilde f_{E}(x_1)=0$ if $x_1\leq \delta_{E}/2$, 
$\tilde f_{E}(x_1)=1$  if $x_1\geq \delta_{E}$, and 
$\tilde f_{E}'(x_1)>0$   if $\delta_{E}/2 < x_1 <  \delta_{E}.$
Let  $g(y_1):= -\delta_1y_1/2$ so that $X_{g} =\mathcal{Z} - Y=-\frac{\delta_1}{2}\partial_{x_1}$. Define  
$$
\mathcal{Z}_{E} := Y + X_{g \tilde{f}_{E}}
$$ 
that satisfies $\mathcal{Z}_E = Y$ on $0< x_1\leq \delta_E/2$ and $\mathcal{Z}_E = \mathcal{Z}$ on $\delta_E\leq x_1 .$ 

We claim that $\mathcal{Z}_E$ is transverse to $W_E$ on $0\leq x_1<\delta_E$ if $E>0$ is fixed sufficiently small.   
We compute 
$$
d H \cdot \mathcal{Z}_{E} = (1-\tilde{f}_{E})d H\cdot Y+\tilde{f}_{E}d  H \cdot \mathcal{Z} + g dH \cdot X_{\tilde{f}_{E}}.
$$
We see that $d H \cdot Y = \frac{1}{2}(y_1^2+y_2^2+ax_1^2+b x_2^2) +  \frac{1}{2}(x_1\partial_{x_1}R + x_2 \partial_{x_2}R)>0$ on $W_E\cap \{|x_1|\leq \delta_E\}$, where $\delta_E:=c\sqrt{E}$
for some $c>0$ sufficiently small independent of $E>0$. This follows from  the fact that $d H \cdot Y|_{W_E\cap\{x_1=0\}}=\frac{1}{2}(y_1^2+y_2^2+bx_2^2)+\frac{1}{2}x_2\partial_{x_2}R(0,x_2)>\frac{E}{2}>0$ for every $E>0$ fixed sufficiently small. The second and third terms are non-negative. This follows from \eqref{dHZ} and the inequality
$g d H \cdot  X_{\tilde{f}_{E}} = \delta_{1}y_{1}^{2}\tilde f_{E}'/2\geq 0.$
The claim is proved.  

The symmetry of  $Y$ with respect to the involution $x_1\mapsto -x_1$ implies that the same construction can be done on $\{x_1\leq 0\}$, and also near every $p_i$ to obtain the desired  Liouville vector field  $X_{E}$, transverse to $W_E$ for every $E>0$ sufficiently small. The Liouville vector field $X_0$ transverse to $\dot W_0$ is obtained in the above construction of $\mathcal{Z}_\varepsilon$ by taking $r\to 0$ near each saddle-center $p_i$ for suitable neighborhoods of the arcs $\gamma_i$.
\end{proof}

\subsection{Finite energy planes asymptotic to a Lyapunov orbit}\label{subsec:holocurves}

Let us consider as before that for suitable symplectic coordinates near any saddle-center the mechanical Hamiltonian admits the form $H =\frac{1}{2}(y_1^2+y_2^2+a x_1^2+bx_2^2)+R(x)$, where $a<0,b>0$ and $R=O(|x|^3)$. For $E>0$ small, let $A_E:= H^{-1}(E) \cap \{|x_1|<c\sqrt{E}\}$, where $c>0$ is given in Proposition \ref{prop:transverseLiouville}, and let 
$$
\lambda= \frac{1}{2}(x_1dy_1-y_1dx_1 + x_2dy_2-y_2dx_2).
$$
In this section, we construct an almost complex structure $J_E$ on $\R \times A_E$ adapted to $\lambda_E:=\lambda|_{A_E}$ and admitting a pair of finite energy planes asymptotic to the Lyapunov orbit $P_{2,E}$ through opposite directions.

Let $\hat H_E := H(\sqrt{E}  \hat x, \sqrt{E}  \hat y)/E=\frac{1}{2}(\hat y_1^2+\hat y_2^2+a\hat x_1^2+b\hat x_2^2)+R(\sqrt{E}\hat x,\sqrt{E}\hat y)/E$ and notice that $H^{-1}(E) \equiv \hat H_E^{-1}(1)$ under the re-scaling $(x,y)=\sqrt{E}(\hat x,\hat y)$. Since
$\hat H_E$ converges in $C^\infty_{\text{loc}}$ to $\hat H_0=\frac{1}{2}(\hat y_1^2+\hat y_2^2+a \hat x_1^2+b \hat x_2^2)$ as $E\to 0^+$, we see that the trajectories on $\hat H_E^{-1}(1)$ locally converges to the trajectories on $\hat H_0^{-1}(1)$ as $E \to 0^+$.  In coordinates $(\hat x,\hat y)$, we consider
\begin{equation}\label{1form}
\hat \lambda= \frac{1}{2}(\hat x_1d\hat y_1 - \hat y_1 d\hat x_1 + \hat x_2 d\hat y_2 -\hat y_2 d\hat x_2)
\end{equation}
which restricts to a contact form $\hat \lambda_E$ on $\hat A_E:= \hat H_E^{-1}(1) \cap \{ |\hat x_1| < c\}$,  for  $E \geq 0$ small.  Notice that $\hat \lambda$ differs from $\lambda$ by a constant factor depending on $E$.

Considering new symplectic coordinates  $(\hat x_1,b^{-1/4}\hat x_2,\hat y_1,b^{1/4}\hat y_2)$ and denoting them again by $(\hat x_1,\hat x_2,\hat y_1,\hat y_2)$, we may assume that 
$$
\hat H_0 = \frac{1}{2}(\hat y_1^2 + \sqrt{b}\hat y_2^2 +a\hat x_1^2+\sqrt{b}\hat x_2^2).
$$
The $1$-form $\hat \lambda$ has the same form \eqref{1form} as before.
The Reeb vector field on $\hat A_E$ is denoted by $\hat R_E$. The Reeb vector field on $\hat A_0$ is given by
$$
\hat R_0 = -(\hat y_1, \sqrt{b}\hat y_2,-a \hat x_1,-\sqrt{b} \hat x_2).
$$ 
Notice that $\hat \lambda_0(\hat R_0)=1$.
Our strategy is to define an almost complex structure $\hat J_0$ on $\R \times \hat A_0$ and find a pair of planes asymptotic to $\hat P_{2,0}\subset \hat H_0^{-1}(1)$ whose projections to $\hat A_0$ lie in $\{\hat x_1=0\}$. Then, for any family of almost complex structures $\hat J_E$ on $\R \times \hat A_E$, $E\geq 0$ small, that smoothly continues $\hat J_0$, we obtain a pair of corresponding $\hat J_E$-holomorphic planes by automatic transversality \cite{wendl2010automatic}. Such planes can now be re-scaled back to the original coordinates $(x,y)$, giving the desired pair of holomorphic planes asymptotic to $P_{2, E}\subset H^{-1}(E).$ 

We use the transverse frame $\{Y_1, Y_2\}=\{j_1 \nabla \hat H_0,j_2 \nabla \hat H_0\},$ where $j_1, j_2$ are as in \eqref{eq_quaternion}, adapted to $\hat A_0$ to define an almost complex structure $\hat J_0$ on $\R \times \hat A_0$. In that case, we have 
$$
Y_1=(\sqrt{b}\hat y_2,-\hat y_1,\sqrt{b} \hat x_2,-a \hat x_1) \quad \mbox{ and } \quad Y_2 = (\sqrt{b} \hat x_2,-a \hat x_1,-\sqrt{b} \hat y_2,\hat y_1).
$$
Since $\rm{span}\{Y_1,Y_2\}$ and $\ker \hat \lambda_0$ are transverse to the Reeb vector field $\hat R_0$, we can project $Y_1$ and $Y_2$ to $\ker \hat \lambda_0$ along $\hat R_0$ to obtain  $\hat Y_1 = Y_1 - \hat \lambda_0(Y_1)\hat R_0$ and $\hat Y_2 = Y_2 - \hat \lambda_0(Y_2)\hat R_0$, both contained in the contact structure $\ker \hat \lambda_0$. Then we take the compatible almost complex structure $\hat J_0$ on $\R \times \hat A_0$  determined by $\hat J_0 \cdot \partial_a = \hat R_0$ and
\begin{equation}\label{def:JandY}
\hat J_0 \cdot \hat Y_2 := \hat Y_1.
\end{equation}
Notice that $d \hat \lambda_0(\hat Y_2,\hat Y_1)>0$. 
We seek for a pair of $\hat J_0$-holomorphic planes, both asymptotic to 
$$
 \hat P_{2,0} = \{\hat x_1= \hat y_1=0,  \sqrt{b} \hat y_2^2+ \sqrt{b} \hat x_2^2=2\},
$$
which approach $\hat P_{2,0}$ through opposite directions, that is through different signs of $\hat y_1$. Taking the local behavior of the flow in $\hat A_0$ into account, we shall construct planes 
which project onto the hemispheres of the 2-sphere $\{\hat x_1=0\} \cap \hat H_0^{-1}(1)$ whose equator is $\hat P_{2,0}$. We consider  cylinders 
$$
\tilde u=(d,u)\colon \R \times \left( \R / 2\pi \Z \right) \to \R \times \hat A_0,
$$ 
of the form
\begin{equation}\label{anzats}
\begin{aligned}
    d=d(s) \quad \mbox{ and } \quad u(s,t) =\left(0,g(s)\cos t, f(s), g(s) \sin t\right),
\end{aligned}
\end{equation}
where $d(s), f(s)$ and $g(s)$ will be determined below. Notice that 
$$
f(s)^2+\sqrt{b}g(s)^2=2,  \forall s,
$$ 
since $u(s,t) \in \hat H_0^{-1}(1), \forall (s,t)$.  Also, $\tilde u$ projects to $\hat x_1=0$. Let us denote by $\pi\colon T\hat A_0 \to \ker \hat \lambda_0$ the projection along $\hat R_0$. The first condition we need to impose on $\tilde u$ to obtain a $\hat J_0$-holomorphic cylinder is the first equation in \eqref{eq_psedo_holom_2}.
Using the ansatz \eqref{anzats}, we compute
$$
\begin{aligned}
\pi u_s & = u_s - \hat \lambda_0(u_s) \hat R_0 = \left(0,g'(s)\cos t,f'(s),g'(s)\sin t\right)\\
\pi u_t & = u_t - \hat \lambda_0(u_t) \hat R_0 =\left(0,-g(s) \sin t,0,g(s) \cos t\right)\\ & \hspace{3cm} -\frac{g(s)^2}{2}\left(-f(s),-\sqrt{b}g(s) \sin t,0,\sqrt{b}g(s) \cos t\right). 
\end{aligned}
$$
We also compute along $u(s,t)$
$$
\begin{aligned}
    \hat Y_1 & =  ( \sqrt{b}g(s)\sin t,-f(s),\sqrt{b}g(s) \cos t,0)\\
    & -\frac{1}{2}(1-\sqrt{b})g(s)f(s)\sin t \left(-f(s),-\sqrt{b}g(s)\sin t,0,\sqrt{b}g(s) \cos t \right)\\
    \hat Y_2 &  = \left(\sqrt{b}g(s)\cos t,0,-\sqrt{b}g(s) \sin t,f(s) \right)\\
    & -\frac{1}{2}(1-\sqrt{b})g(s)f(s)\cos t\left(-f(s),-\sqrt{b}g(s)\sin t,0,\sqrt{b}g(s) \cos t \right). 
\end{aligned}
$$

It follows from the computation above that 
$$
\begin{aligned}
 \cos t \cdot \hat Y_1 - \sin t \cdot \hat Y_2 & = \frac{\sqrt{b}g(s)}{f'(s)}\pi u_s\\
 \sin t\cdot  \hat Y_1 + \cos t \cdot \hat Y_2 & = \frac{f(s)^2+bg(s)^2}{g(s)f(s)}\pi u_t.
\end{aligned}
$$
Imposing that $\hat J_0 \cdot \pi u_s = \pi u_t$, see \eqref{eq_psedo_holom_2}, together with $f(s)^2+\sqrt{b}g(s)^2=2$, and \eqref{def:JandY}, we end up with the following differential equation 
\begin{equation}\label{f(s)}
f' =  -\frac{(2-f^2)f}{f^2+\sqrt{b}(2-f^2)}.
\end{equation}
Given any initial condition  $f(0)\in (0, \sqrt{2} ),$  we integrate \eqref{f(s)} to determine $f(s)$ satisfying $f(s) \to 0^+$ as  $s\to +\infty$  and $f(s) \to \sqrt{2}^-$ as  $s\to -\infty$. The value of $g(s)= \sqrt{(2-f(s)^2) /\sqrt{b}}$ is then determined. Notice that $g(s)^2 \to 2/\sqrt{b}$ as $s\to +\infty$. We also define  
$$
d'(s) := \hat \lambda_0(u_t) = \frac{g(s)^2}{2}  \quad \Rightarrow  \quad d(s) = \int_0^s \frac{g(\tau)^2}{2} d\tau, 
$$
see    \eqref{eq_psedo_holom_2}.
Hence $s=+\infty$ is a positive puncture and $s=-\infty$ is a removable puncture. After removing the puncture at $-\infty$, we obtain a finite energy plane $\tilde u_1=(d_1,u_1)\colon\C \to \R \times \hat A_0$ asymptotic to $\hat P_{2,0}.$

Choosing an initial condition $f(0)\in (-\sqrt{2},0),$ we proceed in the same way to obtain a finite energy plane $\tilde u_2=(d_2,u_2)\colon\C \to \R \times \hat A_0$ asymptotic to $\hat P_{2,0}$ through the opposite direction.  

Such planes are automatically transverse, and their Fredholm index is $1$, see \cite{wendl2010automatic}. Hence if $\hat J_E$ is a smooth family of almost complex structures   on $\hat A_E$ extending $\hat J_0$ on $\hat A_0$, we obtain two families of $\hat J_E$-holomorphic planes $\hat {\tilde u}_{i,E}=(\hat d_{i,E},\hat u_{i,E}), i=1,2,$ asymptotic to the Lyapunov orbit $\hat P_{2,E} \subset \hat A_E$ through opposite directions, and whose projections to $\hat H_E^{-1}(1)$ lie in $\{|\hat x_1|<c\}$ for any $E>0$ sufficiently small.
We can bring such planes back to the original coordinates  $(x,y)= (\sqrt{E}\hat x,\sqrt{E}\hat y)$ in the following way: on the contact structure, we consider the almost complex structure $J_E$ which coincides with $\hat J_E$ under the re-scaling, and define $ u_{i, E}(s,t) = \sqrt{E}\hat u_{i, E}(s,t)$. Since $\lambda$ gets multiplied by $1/E$ in coordinates $(x,y)$, the new real-valued function $d_{i,E}$ is defined as $d_{i,E}(s):=E \hat d_{i,E}(s)$. The new planes $\tilde u_{i,E}=(d_{i,E},  u_{i,E})$ are finite energy $J_E$-holomorphic planes asymptotic to the Lyapunov orbit $P_{2,E}\subset H^{-1}({E})$ through opposite directions. Moreover, $u_{i, E}(\C)$ converges in the Hausdorff topology to the saddle-center at $0\in \R^4$ as $E \to 0^+$. This procedure can be done for every saddle-center. Using an auxiliary metric on $W_E$, the locally defined $J_E$ extends to an almost complex structure on $\R \times W_E.$ We have proved the following proposition.

\begin{proposition}\label{prop_rigid_planes}
    Under the conditions of Proposition \ref{prop:transverseLiouville}, the following holds. For every $E>0$ sufficiently small, there exists an almost complex structure $J_E$ on $\R \times W_E$ adapted to $\lambda_E$ so that the Lyapunov orbit near each $p_i$ is the asymptotic limit of a pair of $J_E$-holomorphic planes $\tilde u_{i,E}=(d_{i,E},u_{i,E}),i=1,2,$ through opposite directions, whose projections to $W_E$ are arbitrarily close to $p_i$ as $E\to 0$. 
\end{proposition}

\subsection{Proof of Theorem \ref{main_theorem}}\label{subsec_proof}

We complete the proof of Theorem \ref{main_theorem} by applying the results of the previous sections. We can assume that the potential $V(x)=V_\epsilon(x)$ is as in Proposition \ref{prop_capping}, where $\epsilon>0$ is sufficiently small. Here, the potential $V$ is changed away from the disk-like compact set $K_0\subset \{V\leq 0\}$ given by the projection of the singular sphere-like subset $S_0\subset H^{-1}(0)$. The Hill region $\{V\leq 0\}$ becomes the union of $K_0$ and $l$ disjoint disk-like compact domains $K_i,i=1,\ldots,l,$ touching $K_0$ at the saddles $v_1,\ldots,v_l\in \partial K_0$. We know from Proposition \ref{prop_weakconv} that if $\epsilon>0$ is fixed sufficiently small, the following conditions hold. If $E>0$ is sufficiently small, then the sphere-like hypersurface $W_E= H^{-1}(E),$ whose projection to the $x$-plane contains $K_0\bigcup\cup_{i=1}^l K_i,$ is weakly convex, the only index-$2$ orbits are the Lyapunov orbits around $p_1,\ldots,p_l$ and every index-$3$ orbit is not linked with any Lyapunov orbit. Since $H$ is a mechanical system,  $W_E$ admits a contact form $\lambda_E$ whose Reeb flow parametrizes the Hamiltonian flow. Hence, we can apply the main result in \cite{Weakconvex}, see Theorem \ref{thm:weakconv}, to obtain a weakly convex foliation $\F_E$ whose binding orbits are the Lyapunov orbits and $l$ index-$3$ orbits in each domain bounded by the rigid planes. In particular, $W_E$ is the union of $l+1$ compact subsets bounded by the pairs of rigid planes.  We know that the foliation $\F_E$ given in Theorem \ref{thm:weakconv} is the projection to $W_E$ of a finite energy foliation in $\R \times W_E$ associated with an almost complex structure $J_E$. By Propositions \ref{prop:transverseLiouville} and \ref{prop_rigid_planes}, we can choose $\lambda_E$ and $J_E$ (perhaps after a $C^\infty$-small perturbation of $J_E$) so that the rigid planes asymptotic to the Lyapunov orbits project arbitrarily close to the saddles $v_1,\ldots,v_l$ as $E \to 0^+$, see Remark \ref{rem_fef}. In particular,  $W_E$ contains a compact subset $S_E$ bounded by the closure of all rigid planes so that its projection to the $x$-plane lies in the domain of the initial potential. In particular, $S_E$ lies inside the initial energy surface $H^{-1}(E)$, and thus $\F_E$ restricted to $S_E$ is the desired weakly convex foliation. This finishes the proof of Theorem \ref{main_theorem}. \qed

\section{Applications}\label{sec:examplesd}\label{sec_applications}

We apply Theorem \ref{main_theorem} to the H\'enon-Heiles potential for energies slightly above $1/6$. We also discuss the existence of weakly convex foliations for decoupled systems, including frozen Hill's lunar problem with centrifugal force, the Stark problem, the Euler problem of two centers, and a chemical reaction model. 

\subsection{The H\'enon-Heiles system}\label{sec_hh}
  In 1964, H\'enon and Heiles \cite{henon1964applicability} proposed the following potential to study the motion of a star in galaxy with an axis of symmetry
$$
V(x_1,x_2)=\frac{1}{2} (x_1^2+x_2^2)+x_1^2x_2-\frac{1}{3}x_2^3.
$$
The dynamics of $V$ is very rich, and we refer to the survey \cite{churchill1979survey}, and \cite{lunsford1972stability, de1998two, churchill1980pathology, ragazzo1994nonintegrability, bahni} for a discussion on periodic orbits, multiple horseshoes, non-integrability, etc. 

The potential $V$ is invariant under the rotation $(x_1,x_2) \mapsto e^{2\pi i/3}(x_1,x_2)$ and under the reflection $(x_1,x_2) \mapsto (-x_1,x_2)$, which means that $V$ exhibits a triangular symmetry ($\mathbb{Z}_3$-symmetry under $\frac{2\pi}{3}$-rotation) and also a reflection symmetry across the vertical axis $x_1=0$. If $0< E <1/6$, the energy surface $H^{-1}(E)$ contains a strictly convex sphere-like component $S_E$, see \cite{salomao2004convex,alexsandro_lens}, which admits a disk-like global surface of section. If $E=1/6$, the energy surface becomes singular and the Hill region $\Omega_{1/6}=\{V\leq 1/6\}$ contains a compact subset $K_{1/6}$ bounded by the triangle $T\subset V^{-1}(1/6)$ whose vertices are the saddle points $v_1=(0,1)$, $v_2=(-\sqrt{3}/2,-1/2)$ and $v_3=(\sqrt{3}/2,-1/2)$, see Figure \ref{fig:hh_proj}. The compact set $K_{1/6}$ is the projection of a singular sphere-like subset $S_{1/6}\subset H^{-1}(1/6)$ that contains three singularities $p_1,p_2$ and $p_3$ of saddle-center type, projecting to $v_1, v_2$ and $v_3$, respectively.

\begin{figure}
\centering
\begin{subfigure}{.5\textwidth}
  \centering
  \includegraphics[width=1\linewidth]{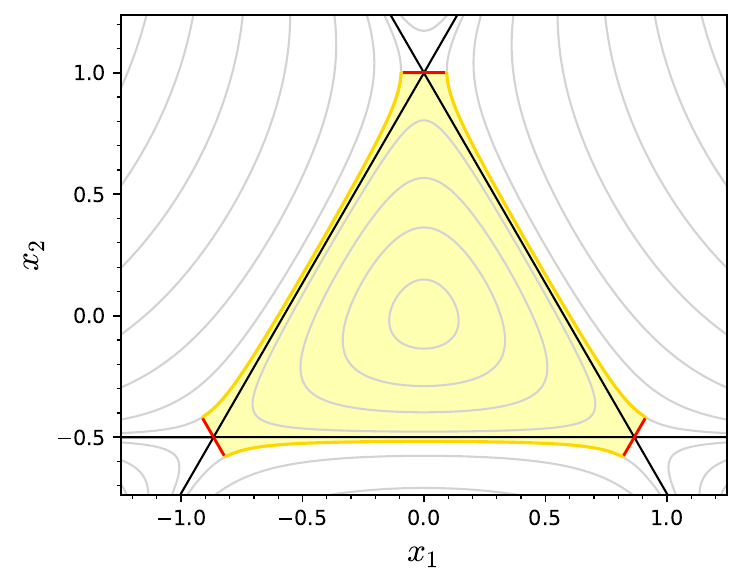}
  \label{fig:sub1a}
\end{subfigure}%
\begin{subfigure}{.5\textwidth}
  \centering
  \includegraphics[width=0.9\linewidth]{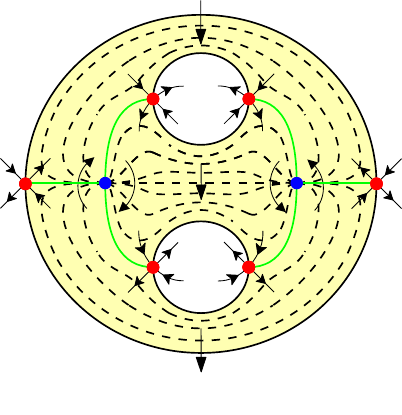}
  \label{fig:sub2a}
\end{subfigure}
\caption{The H\'enon-Heiles system for energies $E>0$ slightly above the critical value $1/6.$ The yellowish region on the left is the projection $K_E\subset \R^2$ of the compact subset $S_E\subset H^{-1}(E)$ to the Hill region. On the right, a cross section of $S_E$ indicates the weakly convex foliation with $l=3$.}
\label{fig:hh_proj}
\end{figure}

\begin{proposition}\label{prop_hh2} The set $S_{1/6} \setminus \{p_1,p_2,p_3\}$ is dynamically convex and the quaternion frame is positive. 
\end{proposition}
\begin{proof}
Notice that $S_{1/6}$ is the Hausdorff limit of $S_E$ as $E\to 1/6^-$. Since $S_E$ is strictly convex for every $E<1/6$, see \cite{salomao2004convex,alexsandro_lens}, $S_{1/6}$ bounds a convex domain of $\R^4$. Let $P\subset S_{1/6}\setminus \{p_1,p_2,p_3\}$ be a periodic orbit. We shall construct a Hamiltonian $\hat H \colon \R^4 \to \R$ whose energy surface $\hat H^{-1}(1/6)$ is regular, star-shaped, and contains a neighborhood of $P$ in $S_E$. In particular, we formally use $\hat H$ to compute the index of $P$. We obtain $\hat H$ by modifying $H$ near $p_1,p_2$ and $p_3$. First, we modify it near $p_1=(0,1)$ and then use the $\Z_3$-symmetry to modify $H$ near $p_2$ and $p_3$.

Let $\delta>0$ be small, and let $f_2 \colon \R \to [0,+\infty)$ be a smooth nondecreasing function satisfying $f_2(t)=0, \forall t\leq 1-\delta,$ and $f(t)=\delta, \forall t\geq 1.$  Let
$$
\hat H (x_1, x_2, y_1, y_2): = H(x_1, x_2, y_1, y_2) + f(x_2) = \frac{|y|^2}{2} + V(x_1,x_2) +f(x_2).
$$
Observe that $\hat H$ coincides with $H$ in $\{x_2 <1-\delta\}$ and $\hat H$ has no critical point in $\{0<x_2<1\}$. Since $f\geq 0$ and $\hat H(x_1,1,0,0)=\frac{3}{2}x_1^2+\frac{1}{6}+\delta > 1/6, \forall x_1\in\R$, we see that $\hat H^{-1}(1/6)$ contains a singular sphere-like component $\hat S_{1/6}$ whose projection $\hat K_{1/6}$ to the $x$-plane lies in $K_{1/6}$ away from the vertex $(0,1)\in T$. Let $Z:=\frac{1}{2}(x_1,x_2,y_1,y_2)$ be the radial vector field. A direct computation gives
$$
\nabla \hat H \cdot Z = \frac{1}{6} + \frac{1}{2}x_1^2x_2 - \frac{1}{6}x_2^3+\frac{1}{2}f'_2(x_2)x_2>0 \quad \mbox{ on } \ \hat S_{1/6} \cap \{0\leq x_2 <1\}.
$$
We can restrict $\hat H$ to a small neighborhood of $\hat S_{1/6}$ and modify it near $p_2$ and $p_3$ in the same manner so that $\hat H$ becomes $\Z_3$-symmetric under $(x,y)\mapsto e^{2\pi i/3}(x,y)$, and $\hat S_{1/6}=\hat H^{-1}(1/6)$ is a regular star-shaped hypersurface. If $\delta>0$ is sufficiently small, then $\hat S_{1/6}$ contains a neighborhood of $P$ in $S_E$. 

Next, we show that the Hessian $\hat{\mathcal{H}}$ of $\hat H$ is positive-definite on $T\hat S_{1/6}$ along  $P$, except perhaps at precisely two points projecting to $\partial K_{1/6}$ in the case $P$ is a brake orbit. Notice that $f=0$ near $P$. Taking $\delta>0$ sufficiently small, it is enough to show that the Hessian $\mathcal{H}$ of $H$ is positive-definite on $TS_{1/6}$ near $P$ in the region projecting to the interior of $K_{1/6}$. From \cite{salomao2004convex}, this condition is equivalent to
$$
G:=2(1/6-V)( V_{x_1x_1}V_{x_2x_2}-V_{x_1x_2}^2) +V_{x_1x_1}V_{x_2}^2+V_{x_2x_2}V_{x_1}^2-2V_{x_1}V_{x_2}V_{x_1x_2}>0
$$
on $K_{1/6}\setminus \partial K_{1/6}$. A direct computation gives
$$
G  = 2(1/6-V)(1-x_1^2-x_2^2).
$$
Since $K_{1/6}\setminus \partial K_{1/6} \subset \{x_1^2+x_2^2<1\},$ we conclude that $G>0$ on $K_{1/6}\setminus \partial K_{1/6}$, and thus $\mathcal{H}$ is positive-definite on $TS_{1/6}$ near $P$ in the region projecting to $K_{1/6} \setminus \partial K_{1/6}.$ If $P$ is a brake orbit, then this condition holds along $P$ except at two points projecting to $\partial K_{1/6}$. 

Summarizing the construction above, we find a Hamiltonian $\hat H$ whose energy surface $\hat H^{-1}(1/6)$ is a star-shaped hypersurface $\hat S_{1/6}$ containing a neighborhood of $P$ in $S_{1/6}$. Moreover, the Hessian $\hat{\mathcal{H}}$ of $\hat H$ is positive-definite along $P$ except possibly at two points projecting to $\partial K_{1/6}$. The Liouville form $\lambda_0$ restricts to a contact form on $\hat S_{1/6},$ whose Reeb flow parametrizes the Hamiltonian flow on $\hat S_E$, and the (generalized) Conley-Zehnder index of $P$ is $\mu(P)$.

Theorem 3.4 in \cite{hofer1998dynamics} states that if $S=H^{-1}(1)\subset \R^4$ for a Hamiltonian $H \colon \R^4 \to \R$ is a strictly convex hypersurface and $P'\subset S$ is a periodic orbit of the Reeb flow of $\lambda_0|_S$, then $\mu(P')\geq 3.$ Investigating its proof, however, we see that it is enough that $S$ is star-shaped, the Hessian of $H$ is $\geq 0$ along $P'$ and positive at some point of $P'$. Such conditions hold for $P\subset \hat S_{1/6}=\hat H^{-1}(1/6)$,  and we conclude that $\mu(P)\geq 3$. 

Finally, we show that the quaternion frame is positive on $S_{1/6} \setminus \{p_1,p_2,p_3\}$. Recall  from \eqref{eq:varangle} that the transverse linearized flow is determined by \eqref{eq:transflow}, and the argument $\theta(t)$ of a solution $\alpha = \alpha_1(t) + i\alpha_2(t)$ satisfies
$$
\dot \theta = (\kappa_{11}+\kappa_{33})\cos^2\theta + 2\kappa_{12} \cos \theta \sin \theta + (\kappa_{22}+\kappa_{33})\sin^2\theta.
$$
We know that for every point of $S_{1/6}\setminus \{p_1,p_2,p_3\}$ projecting to $K\setminus \partial K,$ $\mathcal{H}$ is positive-definite on $TS_{1/6}$. This implies $\dot \theta>0$, see Lemma \ref{lema:convexpositive}. For those points projecting to $\partial K \setminus \{v_1,v_2,v_3\}$ we know from \eqref{kappa22}  that $\kappa_{11} = \kappa_{33}=1,$ $\kappa_{12}=0$, and  
$$
\kappa_{22} = \frac{V_{x_1x_1}V_{x_2}^2-2V_{x_1x_2}V_{x_1}V_{x_2} + V_{x_2x_2}V_{x_1}^2}{V_{x_1}^2+V_{x_2}^2}=0,
$$ 
which implies that $\dot \theta = 1.$ we conclude that the quaternion frame is everywhere positive on $S_{1/6} \setminus \{p_1,p_2,p_3\}$.
\end{proof}

Proposition \ref{prop_hh2} and Theorem \ref{main_theorem} imply the following theorem.

\begin{theorem}\label{thm_HH}
For every $E$ slightly above $1/6,$  there exists a compact subset $S_E \subset H^{-1}(E)$ admitting a weakly convex foliation as in Theorem \ref{main_theorem} with $l=3$, see Figure \ref{fig:hh_proj}. The binding orbits are the three Lyapunov orbits near the saddle-centers and an index-$3$ orbit projecting near $K_{1/6}$. In particular,  $S_E$ contains infinitely many periodic orbits, and each Lyapunov orbit admits infinitely many homoclinic orbits or heteroclinic orbits to other Lyapunov orbits. 
\end{theorem}

\subsection{Decoupled mechanical systems} \label{subsec:decoupled}

We present examples of decoupled Hamiltonian systems that admit weakly convex transverse foliations.  Although Theorem \ref{main_theorem} applies to such systems, the foliations in some cases are obtained from gradient flow lines. We discuss the frozen Hill's lunar problem with centrifugal force, the Stark problem, the Euler problem of two centers, and a chemical reaction model.

Let  $H=H_1(x_1,y_1)+H_2(x_2,y_2)$ be the sum of two one-degree-of-freedom mechanical Hamiltonians 
\begin{equation}\label{H12}
H_1(x_1,y_1)=\frac{y_1^2}{2} + V_1(x_1) \quad \mbox{ and } \quad H_2(x_2,y_2)= \frac{y_2^2}{2} + V_2(x_2),
\end{equation}
where $V_1, V_2$ are smooth one-dimensional potentials with nondegenerate critical points. The trajectories of $H^{-1}(E)$ project to trajectories of $H_1^{-1}(E_1)$ and $H_2^{-1}(E_2)$, where $E=E_1+E_2$. The following proposition implies that $H$ is weakly convex, i.e., all of its non-trivial periodic orbits have index $\geq 2$.

\begin{proposition}\label{prop_decoupled1} Let $H$ satisfy the conditions above, and let $\gamma(t)=(\gamma_1(t),\gamma_2(t)) \in \R^4$ be a simple non-trivial periodic orbit. Then the following statements hold:
   \begin{itemize}
   \item[(i)] If $\gamma_1(t)\equiv(\bar x_1,0)$ is an equilibrium of $H_1$ corresponding to a nondegenerate minimum of $V_1$, and $\gamma_2(t)$ is a non-trivial periodic orbit of $H_2$, then $\mu(\gamma) \geq 3.$ In particular, its rotation number is $>1$.
   \item[(ii)] If $\gamma_1(t)\equiv(\bar x_1,0)$ is an equilibrium point of $H_1$ corresponding to a nondegenerate maximum of $V_1$, and $\gamma_2(t)$ is a non-trivial periodic orbit of $H_2$, then $\gamma$ is a hyperbolic orbit and $\mu(\gamma) = 2.$ In particular, its rotation number is $1$.
   \item[(iii)] If $\gamma_1$ and $\gamma_2$ are non-trivial periodic orbits of $H_1$ and $H_2$, respectively, then $\rho(\gamma)$ is an integer $\geq 2$. In particular, $\mu(\gamma) \geq 3$.   
   \end{itemize} 
The same conclusions hold in (i) and (ii) for $H_1$ interchanged with $H_2$. 
\end{proposition}

\begin{remark}
    Proposition \ref{prop_decoupled1} easily generalizes for potentials $V_1$ and $V_2$ with degenerate critical points. In particular, the dynamics of any decoupled Hamiltonian  $H=(y_1^2+y_2^2)/2+V_1(x_1)+V_2(x_2)$ restricted to a regular energy surface is weakly convex.
\end{remark}

\begin{proof}[Proof of proposition \ref{prop_decoupled1}]
The conclusions can be obtained by summing up the rotation numbers on the symplectic planes $x_1y_1$ and $x_2y_2$. For clarity, however, we stick to the geometric definition of $\mu(\gamma)$ and compute it using the quaternion frame $\{X_1,X_2\}$. 

First, consider cases (i) and (ii). We have
$X_1|_\gamma = g\cdot (y_2\partial_{x_1} + V_2'(x_2) \partial_{y_1})$ and $X_2|_\gamma = g\cdot (V_2'(x_2)\partial_{x_1} - y_2\partial_{y_1})$, where $g=(y_2^2+V_2'(x_2)^2)^{-1/2}.$ Notice that the transverse plane $\Pi_1:=\text{span}\{\partial_{x_1},\partial_{y_1}\}|_\gamma$ is invariant by the flow. Also, the frame $\{X_1,X_2\}|_\gamma$ has opposite orientation with respect to the frame $\{\partial_{x_1},\partial_{y_1}\}|_\gamma$. 

Since $\gamma$ is simple, $\gamma_2$ is also simple. The winding number of $\gamma_2'=(y_2,-V_2')\in \R^2$ is $+1$ in the clockwise direction. Hence, the winding number of $X_1$ and $X_2$ on $\Pi_1$ is $+1$ in the counterclockwise direction. Since the linearized flow on $\Pi_1$ is given by  
$$
\dot z = \left(\begin{array}{cc}0 & 1 \\ -V_1''(\bar x_1) & 0 \end{array} \right) z, \quad z=(\alpha_1 \ \alpha_2)^T\equiv \alpha_1 \partial_{x_1} +\alpha_2\partial_{y_1},
$$ we see that if $V''_1(\bar x_1)>0,
$  then the variation in the argument of every linearized solution on the frame  $\{\partial_{x_1},\partial_{y_1}\}|_\gamma$ is $<0$ and thus is greater than  $2\pi$ on $\{X_1, X_2\}|_\gamma$. Hence, the index of $\gamma$ is $\geq 3$. If $V''_1(\bar x_1)<0$, then the variations in the argument of all non-trivial solutions on $\{\partial_{x_1},\partial_{y_1}\}|_\gamma$ give an interval containing $0$ in its interior, and thus, an interval containing $2\pi$ in its interior on the frame $\{X_1, X_2\}|\gamma$. Hence, $\gamma$ is an index-$2$ hyperbolic orbit. 

In case (iii), we use a more direct proof. Since $H$ is decoupled, we have $\rho(\gamma) = \rho(\gamma_1) + \rho(\gamma_2),$ where $\rho(\gamma_i)$ is the rotation number of $\gamma_i$ in each symplectic factor. Since  $\gamma_1$ and $\gamma_2$ are non-trivial periodic orbits, we have $\rho(\gamma_i)\geq 1, i=1,2.$ Their velocity vectors are invariant under the flow in each factor and thus $\rho(\gamma_i)$ is an integer $\geq 1$. Hence $\rho(\gamma)$ is an integer $\geq 2$. In particular, $\mu(\gamma) \geq 3$. 
\end{proof}

Proposition \ref{prop_decoupled1} tells us that an index-$2$ orbit of $H$ has the form $\{(\bar x_1 ,0)\} \times \gamma_2$ or $\gamma_1 \times \{(\bar x_2 ,0)\}$, where $\bar x_i$ is a maximum of $V_i$ and $\gamma_j$ is a non-trivial periodic orbit of $H_j,j\neq i.$ 

Notice that every critical point of 
$$
V(x_1,x_2):=V_1(x_1)+V_2(x_2)
$$ 
has the form $(\bar x_1,\bar x_2)$, where $\bar x_i$ is a critical point of $V_i$. 
Assume that $0$ is a critical value of $H$ satisfying the conditions of Theorem \ref{main_theorem}, i.e., the Hill region $\{V\leq 0\}$ has a compact disk-like  subset $K_0\subset \R^2$ whose boundary $\partial K_0\subset V^{-1}(0)$ contains precisely $l\geq 1$ critical points $v_1,\ldots,v_l$ of $V$, which are all saddles of $V$, and $V|_{K_0\setminus \partial K_0}<0.$ Let $S_0\subset H^{-1}(0)$ be the sphere-like singular subset projecting to $K_0$, and let $p_i=(v_i,0)\in \R^4$ the corresponding saddle-center of $H$. 

Suppose that the saddle $v_i=(\bar x_{i,1},\bar x_{i,2})$ is such that $\bar x_{i,1}$ is a maximum of $V_1$ and $\bar x_{i,2}$ is a minimum of $V_2$. Then 
$p_i=(v_i,0)\in \R^4$ is a saddle-center and the Lyapunov orbit near $p_i$, for $E>0$ sufficiently small, has the form $\{(\bar x_{i,1},0)\}\times \gamma_{i,2,E}$, where $\gamma_{i,2,E}(t)$ is a non-trivial periodic orbit of $H_2$ near $(\bar x_{i,2},0)\in \R^2$. In particular, its projection to the $x$-plane is vertical. A similar statement holds if $\bar x_{i,1}$ is a minimum of $V_1$ and $\bar x_{i,2}$ is a maximum of $V_2$, and we obtain a Lyapunov orbit $\gamma_{i,1,E}\times \{(\bar x_{i,2},0)\}$ whose projection to the $x$-plane is horizontal.

For every $E>0$ small, we denote by $K_E$ the disk-like compact subset of the Hill region $\{V\leq E\}$ bounded by simple arcs of $V^{-1}(E)$ near $\partial K_0$ and the horizontal and vertical projections of the $l$ Lyapunov orbits near the saddle-centers. Denote by $S_E$ the subset of $H^{-1}(E)$ projecting to $K_E$. Let $\pi_i(x_1,x_2):=x_i,i=1,2,$ be the projection on each factor.

\begin{proposition}\label{prop_dec_dyn_conv}
    If $\pi_i(K_0 \setminus \partial K_0)\subset \R$ contains no maximum of $V_i$, i=1,2, then 
    \begin{itemize}
        \item[(i)] $S_0 \setminus \{p_1,\ldots,p_l\}$ is dynamically convex.
        \item[(ii)] For every $E>0$ sufficiently small, every periodic orbit on $S_E,$ which is not a Lyapunov orbit around $p_i,\ldots,p_l,$  has index $\geq 3$.
    \end{itemize}
\end{proposition}

\begin{remark}
The dynamical convexity in Proposition \ref{prop_dec_dyn_conv} also holds if $V_1$ and $V_2$ have degenerate critical points.
\end{remark}
\begin{proof}[Proof of Proposition \ref{prop_dec_dyn_conv}]
Assume by contradiction that $\gamma=\{(\bar x_1,0)\}\times \gamma_2$ is an index-$2$ periodic orbit in $S_0\setminus \{p_1,\ldots,p_l\}$, where $\bar x_1$ is a maximum of $V_1$ and $\gamma_2$ is a non-trivial periodic orbit of $H_2$. Since $V(\bar x_1, \pi_2 (  \gamma_2(t))) = V_1(\bar x_1) + V_2(\pi_2(\gamma_2(t)))\leq 0, \forall t,$ we see that  $V(\bar x_1, \pi_2(\gamma_2))$ attains a minimum $<0$ and thus $(\bar x_1,\pi_2(\gamma_2))$ contains points in the interior of $K_0$. Hence $\bar x_1$ is in the interior of $\pi_1(K_0\setminus \partial K_0),$ a contradiction. The same holds if $\gamma=\gamma_1(t)\times \{ (\bar x_2,0)\}$, where $\bar x_2$ is a maximum of $V_2$ and $\gamma_1$ is a non-trivial periodic orbit of $H_1$. This proves (i). 

The proof of (ii) follows the same argument as in (i) with the remark that every critical point of $V_1$ or $V_2$ is non-degenerate and thus isolated, and thus no index-$2$ periodic orbit on $S_E$ exists other than the Lyapunov orbits around $p_1,\ldots,p_l$. 
\end{proof}

We shall use Proposition \ref{prop_dec_dyn_conv} to obtain a weakly convex foliation on $H^{-1}(E)$.

\begin{theorem}\label{thm:decou}
    Assume that conditions in Proposition \ref{prop_dec_dyn_conv} hold. Then, for every $E>0$ sufficiently small, $S_E$ admits a weakly convex foliation whose binding is formed by the Lyapunov orbits near $p_1,\ldots,p_l$ and an index-$3$ orbit in $S_E \setminus \partial S_E$. 
\end{theorem}

\subsection{Frozen Hill's lunar problem with  centrifugal force}

This problem is a mechanical system introduced in \cite{CFvK17} whose  Hamiltonian in canonical coordinates $(q=q_1+iq_2,p=p_1+ip_2)$  is given by
$$
H_0(q,p)= \frac{|p|^2}{2} -\frac{1}{|q|} -\frac{3}{2}q_1^2 - \frac{1}{2}(q_1^2+q_2^2).
$$
Using Levi-Civita coordinates $(q=v^2,p=\frac{u}{2\bar v})$, we obtain the regularized Hamiltonian
$$
K(v,u) =4|v|^2 (H_0 +c)=\frac{|u|^2}{2}-2|v|^2\left( 3(v_1^2-v_2^2)^2+|v|^4-2c\right)-4,
$$
and $K^{-1}(0)$ doubly covers the regularized $H_0^{-1}(-c)$. Let us assume that $c>0$. Defining 
$u=c^{3/4}y$ and $v=c^{1/4}x$, we obtain the parameter-free Hamiltonian 
\begin{equation}\label{potential_frozen}
H(x,y) = \frac{1}{c^{3/2}}(K(v,u)+4)= \frac{|y|^2}{2} +4x_1^2-8x_1^6+4x_2^2-8x_2^6,
\end{equation}
and $H^{-1}(4/c^{3/2})$ still doubly covers the regularized $H_0^{-1}(-c).$

The potentials $V_1(x_1) = 4x_1^2-8x_1^6$ and $V_2(x_2) = 4x_2^2-8x_2^6$ have a minimum at $0$ and two nondegenerate maxima at $\pm 6^{-1/4}$. Hence $V(x_1,x_2) = V_1(x_1)+V_2(x_2)$ has four saddle-points $v_{1,2}=(\pm 6^{-1/4}, 0)$ and $v_{3,4}=(0, \pm 6^{-1/4})$ at level $e_*=\frac{4}{3}\sqrt{\frac{2}{3}}$. The sphere-like singular subset $S_{e_*}\subset H^{-1}(e_*)$ contains four saddle-center equilibrium points $p_i$ projecting to $v_i$, $i=1,2,3,4.$ The projection of $S_{e_*}$ to the $x$-plane is the disk-like compact domain $K_{e_*}$ with $\Z_4$-symmetry under the $\frac{\pi}{2}$-rotation, see Figure \ref{fig:frozen}. The $x_1$ and $x_2$-projections of $K_{e_*}$ lie in the interval $(-6^{1/4},6^{1/4})$, and hence contain no maximum point of $V_1$ and $V_2,$ respectively. For $E-e_*>0$ small, denote by $K_E\subset \R^2$ the disk-like subset of the Hill region $\{V\leq E\}$ near $K_{e_*}$ as in section \ref{subsec:decoupled}. Denote by $S_E$ the subset of $H^{-1}(E)$ projecting to $K_E$. By  Theorem \ref{thm:decou}, we obtain the following application.

\begin{proposition}
    For every $E-e_*>0$ sufficiently small, the energy surface $H^{-1}(E)$ admits a weakly convex foliation on $S_E$ with $l=4$. The binding orbits are the Lyapunov orbits around $p_1,\ldots,p_4$  and an index-$3$ orbit projecting to $K_E  $.  
\end{proposition}

\begin{figure}
\centering
\begin{subfigure}{.5\textwidth}
  \centering
  \includegraphics[width=1\linewidth]{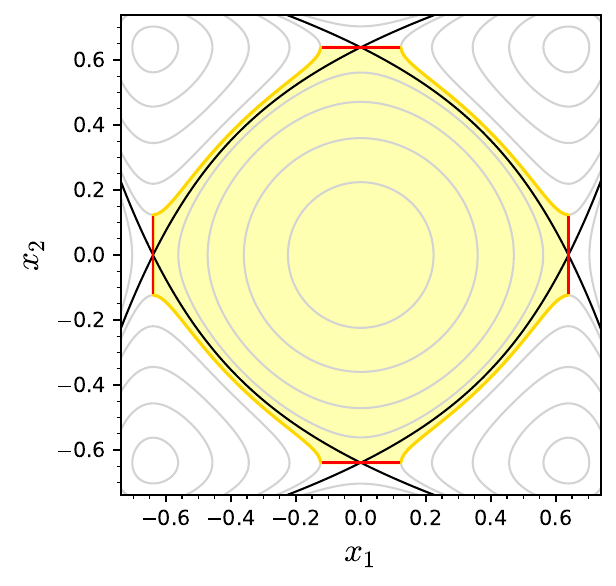}
  \label{fig:sub1}
\end{subfigure}%
\begin{subfigure}{.5\textwidth}
  \centering
  \includegraphics[width=0.9\linewidth]{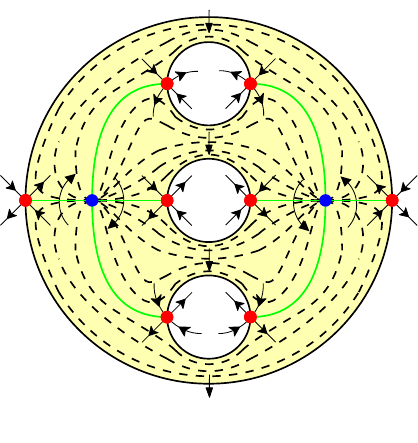}
  \label{fig:sub2}
\end{subfigure}
\caption{The frozen Hill's lunar problem with centrifugal force for energies above the critical value. The projection $K_E$ of $S_E\subset H^{-1}(E)$ to the Hill region (left). $S_E$ is foliated by a weakly convex foliation with $l=4$ (right).}
\label{fig:frozen}
\end{figure}

\subsection{The Stark problem}
   
In \cite{CFvK17}, Cieliebak, Frauenfelder, and van Koert introduced a large class of Hamiltonian systems called Stark systems, which model the motion of an electron attracted by a proton subject to an external electric field.   A planar Stark system is a  mechanical system whose potential has the form 
$U= -\frac{1}{|q|} + U_0(q),$ where $U_0$ is smooth near $q=0$. We assume that   there exists $a>0$ such that $U_0(\nu^{2a} q) = \nu U_0(q)$ for every $q$ and $\nu>0$. 
The frozen Hill's lunar problem studied in the previous section is a Stark system with $a=1/4$. As before, the Hamiltonian $H_0=|p|^2/2 + U(q)$ can be regularized using adapted Levi-Civita coordinates $q=c^{1/2}x^2, p=c^{1/2}\frac{y}{2\bar x}$, $c>0$, and becomes 
$$
H(x,y) = \frac{|y|^2}{2} + V(x),\quad \mbox{ where } \quad 
V(x) = 4|x|^2(1+U_0(x^2)).
$$
The energy surface  $H^{-1}(4/c^{3/2})$ doubly covers the regularized $H_0^{-1}(-c)$.

  The problem studied by Stark  \cite{Stark1914}, known as the Stark problem, is the case of a constant electric field, that is $U_0(x) =  - \varepsilon x_1, \varepsilon >0$ ($a=1/2$). Many researchers have extensively investigated this problem in a variety of contexts. We refer the reader to  \cite{Isayev72, Lantoine09} and \cite{Barratt83, Hezel92}, where the Stark problem is studied in orbital and quantum mechanics, respectively.  We also refer to  \cite{Stark_Urs} in which the author studies the Stark problem in the flavor of symplectic geometry.
    
  The potential $V$ of the Stark problem is given by
$$
V(x_1,x_2)=V_1(x_1)+V_2(x_2) =4(x_1^2-\varepsilon x_1^4)+   4(x_2^2+  \varepsilon x_2^4). 
$$
Notice that $V$ admits a minimum at $(0,0)\in V^{-1}(0)$ and two saddles at $v_{1,2}=(\pm (2\varepsilon)^{-1/2},0)$ with critical value
$e_*=\varepsilon^{-1}>0$. The sub-level set $\{V\leq e_*\}$ contains a disk-like compact domain $K_{e_*}\subset \{V\leq e_*\}$  whose boundary contains the saddles $v_1$ and $v_2$. The compact $K_{e_*}$ is the projection of a singular sphere-like subset of $H^{-1}(e_*)$ with two saddle-centers at $p_i=(v_i,0),i=1,2.$ For each $E-e_*>0$, denote by $S_E$ the subset of $H^{-1}(E)$ whose projection to the $x$-plane is the disk-like region $K_E$ bounded by the projections of the Lyapunov orbits around $p_1,p_2$ and arcs in $V^{-1}(E)$. 

The hypotheses of Proposition \ref{prop_dec_dyn_conv} are trivially satisfied. However, since $V_2$ has a unique critical point (a nondegenerate global minimum), it is easier to construct the transverse foliation on $H^{-1}(E)$ using the gradient flow lines of $H_1$. For $0< E< e_*,$ $H^{-1}(E)$ contains a regular sphere-like hypersurface $S_E$ admitting an open book decomposition whose pages are disk-like global surfaces of section. The binding orbit is the one projecting to $x_1=y_1=0$, and the planes are the pre-images of the gradient flow lines of $H_1$ under the projection to the $x_1y_1$-plane, see Figure \ref{fig:stark}. The singular sphere-like subset $S_{e_*}$ has two singularities at $p_1,p_2$. A similar open book exists for $S_{e_*}.$ For $E>e_*$, $H^{-1}(E)$ admits a weakly convex foliation whose Lyapunov orbits are part of the binding, as in the following proposition. 

\begin{figure}[t]
\centering
\begin{subfigure}{.5\textwidth}
  \centering
  \includegraphics[width=1\linewidth]{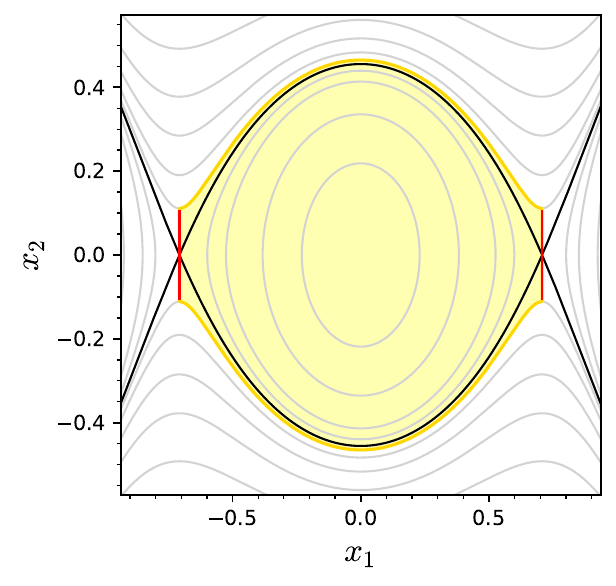}
  \label{fig:sub1c}
\end{subfigure}%
\begin{subfigure}{.5\textwidth}
  \centering
  \includegraphics[width=1\linewidth]{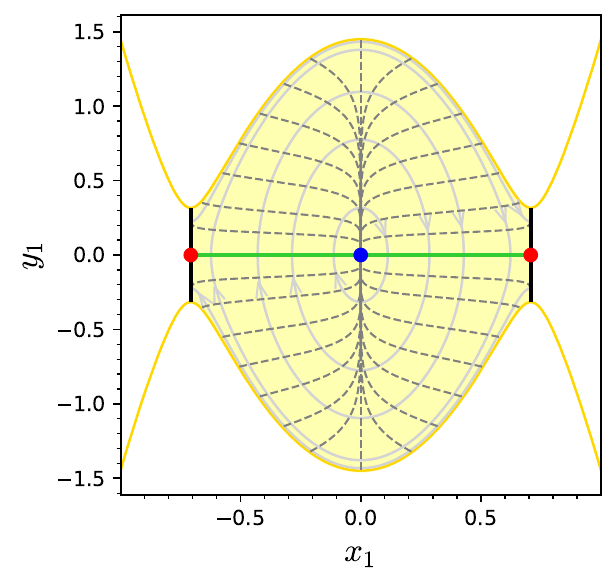}
  \label{fig:sub2b}
\end{subfigure}
\caption{The Stark problem for energies above the critical value. The energy surface $H^{-1}(E)$ contains a sphere-like hypersurface $S_E$ admitting a weakly convex foliation with $l=2$, see Figure \ref{fig:weak}. The projection $K_E$ of $S_E$ to the Hill region (left). The binding is formed by the Lyapunov orbits and the orbit projecting to $(x_1,y_1)=(0,0)$; its index is $\geq 3$. The regular leaves project to the gradient flow lines of $H_1$ (right).}
\label{fig:stark}
\end{figure}

\begin{proposition} \label{prop:Starkconvex}
For every $E>e_*$, the energy surface $H^{-1}(E)$ admits a weakly convex foliation on $S_E$. The binding orbits are the two Lyapunov orbits around $p_1,p_2$  and an index-$3$ orbit projecting to $K_E,$ and the regular leaves project to the gradient flow lines of $H_1$.  
 \end{proposition}

\begin{remark} In general, for a decoupled potential $V=V_1(x_1)+V_2(x_2)$, where $V_2$ is a coercive function with a single critical point (a global minimum), a transverse foliation whose regular leaves are planes and cylinders can be obtained from the gradient flow lines of $H_1=y_1^2/2+V_1(x_1)$. This follows from the fact that the trajectories of $H_2=y_2^2/2+V_2(x_2)$ are embedded circles around the minimum, and the inverse image of the gradient flow lines of $H_1$ project to $(x_2,y_2)$ as annuli between those circles. For certain energy surfaces, they may determine open book decompositions whose pages are global surfaces of section. For instance, see \cite{ kim2024transverse}.
\end{remark}

\subsection{The Euler problem of two centers} The Euler problem of two fixed centers in the plane describes the motion of a massless body (the satellite)  under the influence of two fixed bodies (the primaries)  which attract the satellite according to  Newton's law of gravitation. See \cite{OM, WDR}.  Fix $0<\mu<1$ and let $E=(-1,0)\in \R^2$ and $S=(1,0)\in \R^2$ denote the primaries with respective mass $\mu$ and $1-\mu$. The potential $U_\mu$ of central force fields centered at $S$ and $E$ is given by
$$
U_\mu(x_1,x_2) = -\frac{\mu}{\sqrt{(x_1+1)^2+x_2^2}} -\frac{1-\mu}{\sqrt{(x_1-1)^2+x_2^2}},
$$
and the mechanical Hamiltonian is $H_\mu=\frac{y_1^2+y_2^2}{2} + U_\mu(x_1,x_2)$. Euler was the first to notice in 1760 that this problem is completely integrable \cite{Euler1, Euler21, Euler22}.   
In coordinates $(x_1,x_2,y_1,y_2)\in \R^4$,  one easily checks the existence of a unique rest point, given by $p_c=(\bar x_1,0,0,0)$, where
\begin{equation*}\label{eq_xbarra}
-1<\bar x_1 = \frac{2\mu-1}{1+2\sqrt{\mu-\mu^2}} <1.
\end{equation*} 
This critical point is  a saddle-center and  the corresponding critical value is equal to
\begin{equation}\label{ccrit}
c_{\rm crit}:=H_\mu(p_c)=-\frac{1}{2}-\sqrt{\mu-\mu^2}.
\end{equation}

We change coordinates $(x_1,x_2)$ to elliptic coordinates $(u,v)\in \R \times \R / 2\pi \Z,$ where 
\begin{equation*}\label{eq_change}
x_1  = \cosh u \cos v, \ \ \ \ x_2  = \sinh u\sin v,
\end{equation*}
and momenta $(y_1,y_2)$  to $(p_u,p_v)$ satisfying
$$
\begin{aligned}
y_1= \frac{p_u\sinh u \cos v - p_v\cosh u \sin v}{\cosh ^2 u - \cos^2 v},\\
y_2= \frac{p_u\cosh u \sin v + p_v\sinh u \cos v}{\cosh ^2 u - \cos^2 v}.
\end{aligned}
$$
The change of coordinates $(x,y,p_x,p_y)\mapsto (u,v,p_u,p_v)$ is symplectic up to a constant factor, and we obtain the new Hamiltonian
$$
H_\mu = \frac{1}{\cosh^2 u- \cos^2 v} \left(\frac{p_u^2 + p_v^2}{2} -\mu(\cosh u - \cos v)-(1-\mu)(\cosh u + \cos v) \right).
$$
The collisions to $S$ and $E$ correspond to $(u,v)=(0,0)$ and $(u,v)=(0,\pi)$, respectively, where the denominator in the expression above vanishes. In order to regularize such collisions, we fix $c\in \R$ and consider the regularized Hamiltonian
$$
\begin{aligned}
K_{\mu,c}   :=& \; (H_\mu - c)(\cosh^2 u- \cos^2 v)\\
  =& \;  \frac{p_u^2}{2}-\cosh u -c \cosh^2 u +\frac{p_v^2}{2} +(2\mu-1) \cos v +c\cos^2 v,
\end{aligned}
$$ where the energy $c$ is now seen as a parameter.

Except for the collision points, the dynamics on
$K_{\mu,c}^{-1}(0)$ corresponds to the dynamics on $H_\mu^{-1}(c)$ via a double covering map which identifies antipodal points
$(u,v,p_u,p_v) \sim - (u,v,p_u,p_v).$

 Denote
$\Sigma_{\mu,c}  := K_{\mu,c}^{-1}(0).$
We consider $\Sigma_{\mu,c}$ lying inside $\R \times \R / 2\pi \Z \times \R \times \R$ with coordinates $(u,v,p_u,p_v)$.  The Hamiltonian $K:=K_{\mu,c}$  is decoupled, that is 
\begin{equation}\label{eq_ham_2}
    K = K_{1}(u,p_u)+K_{2}(v,p_v),
\end{equation} 
where
$
K_{1}(u,p_u) := \frac{p_u^2}{2}-\cosh u -c \cosh^2 u
$
and
$
K_{2}(v,p_v)  := \frac{p_v^2}{2} +(2\mu-1) \cos v +c\cos^2 v
$
are first integrals of motion. Denote by $V_1(u)=-\cosh u - c \cosh^2 u$ and $V_2(v)=(2\mu-1)\cos v+c \cos^2 v$ the respective potentials of $K_1$ and $K_2$.

We consider negative energies $c<0,$ so that the motions are bounded. If $c<c_{\rm crit}$, then $\Sigma_{\mu,c}$ contains two dynamically convex sphere-like components admitting disk-like global surfaces of section. If $c_{\rm crit} < c< 0,$ then $\Sigma_{\mu,c}$ is diffeomorphic to $S^1 \times S^2$  corresponding to the connected sum of the lower components. We consider the energy values $c_1<c_2$ given by 
$$ 
-1<c_{\rm crit} <  c_1:=-\frac{1}{2} < c_2   := - \big| \mu - \frac{1}{2} \big| \leq 0.$$ Notice that $ c_2 = 0$ only for $\mu=1/2.$

For $c\leq c_1$,  $K_1$ has a unique critical point at $(0,0)\in K_1^{-1}(-1-c)$, which is a global minimum. The dynamics of $K_1$ is that of a center about $(0,0)$. For $c_1<c<0,$ $K_1$ has three critical points, a saddle at $(0,0)\in K_1^{-1}(-1-c)$ and two symmetric minima on the $u$-axis.

For $c< c_2,$ $K_2$ has $4$ critical points in $\R / 2 \pi \Z \times \{0\} \subset \R / 2\pi \Z \times \R$, two local minima at $(0,0)$ and $(\pi,0)$, and two saddles with the same critical value $-\frac{(2\mu-1)^2}{4c}$. The saddles correspond to index-$2$ hyperbolic periodic orbits, and the minima correspond to elliptic periodic orbits with index $\geq 3$. For $c_2 \leq c <0,$ $K_2$ has only $2$ critical points at $(0,0)$ and $(\pi,0)$, a minimum and a saddle, depending on the sign of $2\mu -1$. The critical value of the saddle is $|2\mu-1|+c.$ These orbits correspond to an index-$2$ hyperbolic orbit and an elliptic orbit with index $\geq 3$. Their projections to the $up_u$-plane are simple orbits enclosing the origin. 

Lifting the gradient flow lines of $K_2$ to the energy surface $\Sigma_{\mu,c} = K_{\mu,c}^{-1}(0) \simeq S^1 \times S^2$, we obtain two different types of transverse foliations: 
\begin{itemize}
    \item[I.] If $c_{\rm crit} < c<c_2,$ then $\Sigma_{\mu,c}$ admits two subsets $S_E, S_E'$ admitting a weakly convex foliation with $l=2$. The index-$2$ orbits project to the continuation of the Lyapunov orbit under the $2:1$ regularization map. The other binding orbits project to collision-brake orbits in $x_2=0$. 
    \item[II.] If $c_2\leq c <0$, the Lyapunov orbit collapses to one of the collision-brake orbits, which becomes hyperbolic after bifurcation. $\Sigma_{\mu,c}$ has only one component that admits a weakly convex foliation with $l=2$. The two binding orbits project to collision-brake orbits under the double-covering map. 
\end{itemize}
    
    The bifurcation in the $up_u$-plane at level $c_1$ is transverse to the foliation and thus does not alter its type.

\begin{figure}[t]
\centering
\begin{subfigure}{.5\textwidth}
  \centering
  \includegraphics[width=1\linewidth]{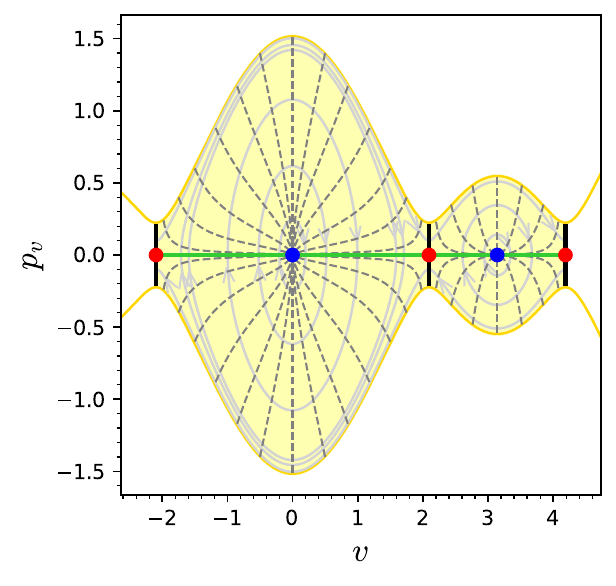}
  \label{fig:sub1c_2}
\end{subfigure}%
\begin{subfigure}{.5\textwidth}
  \centering
  \includegraphics[width=1\linewidth]{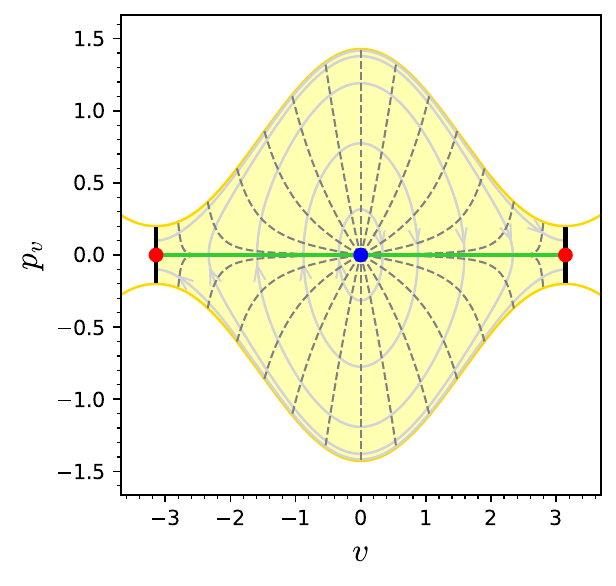}
  \label{fig:sub2b_2}
\end{subfigure}
\caption{The Euler problem for $\mu=1/4$ and $c_{\rm crit} < c=-1/2 < c_2$ (left),  and  $c_2< c=-1/5<0$ (right). Each chamber admits a weakly convex foliation with $l=2$. The regular leaves project to the gradient flow lines of $K_2$.}
\label{fig:euler}
\end{figure}

\subsection{Chemical reaction model}

In \cite{Chemical}, the authors study a chemical reaction described by a mechanical system whose potential is
\begin{equation}\label{pome}
V(x_1,x_2) =V_1(x_1)+V_2(x_2) = -\frac{1}{2}\alpha x_1^2 + \frac{1}{4} \beta x_1^4 + \frac{1}{2} x_2^2,
\end{equation}
where $\alpha \beta >0.$ The potential $V$ has precisely three critical points at $v_0=(0,0)$  and $v_{1,2}=(\pm (\alpha/\beta)^{1/2},0)$. 
The corresponding critical values are $0$ and $-\frac{ \alpha^2}{4 \beta},$ respectively. We split it into two cases:
\begin{itemize}
\item[I.] $\alpha,\beta>0$: the critical level $H^{-1}(e_*)$,  $e_*=0$, contains two singular sphere-like subsets $S_0,S_0'$ with a singularity at the saddle-center $(v_0,0)\in \R^4$. Their projections to the $x$-plane are compact disk-like domains touching each other at  $v_0$. For $E>0$, the sphere-like hypersurface $H^{-1}(E)$ admits a weakly convex foliation whose leaves project to the gradient flow lines of $H_1$. This foliation is called a $3-2-3$ foliation, and the Lyapunov orbit at $x_1=0$ is one of the binding orbits. 

\item[II.] $\alpha,\beta<0$: the critical level $H^{-1}(e_*)$, $e_*=-\alpha^2/(4\beta)$, contains a singular sphere-like subset with two singularities at $(v_1,0)$ and $(v_2,0)$. Its projection to the $x$-plane is a compact disk-like domain with two singularities at $v_1,v_2$.  For $E-e_*>0$, the energy surface also admits a weakly convex foliation   whose leaves project to the flow lines of $H_1$. The binding includes the Lyapunov orbits at $x_1=\pm(\alpha/\beta)^{1/2}$. 
\end{itemize}

\begin{figure}[!ht]
\centering
\begin{subfigure}{.5\textwidth}
  \centering
  \includegraphics[width=1\linewidth]{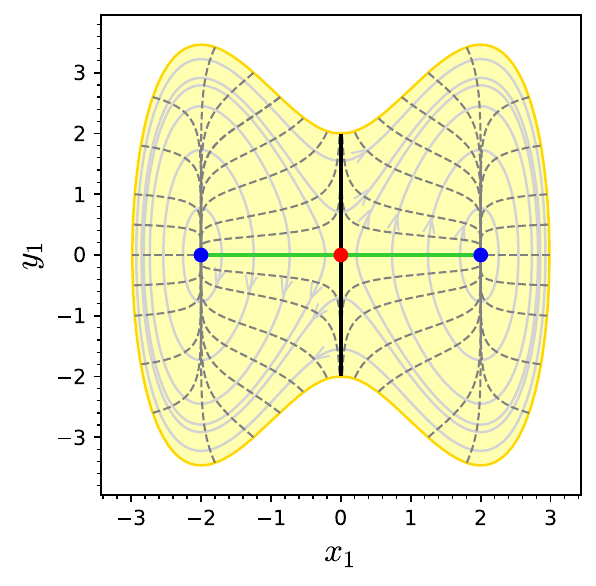}
  \label{fig:sub1d}
\end{subfigure}%
\begin{subfigure}{.5\textwidth}
  \centering
  \includegraphics[width=1\linewidth]{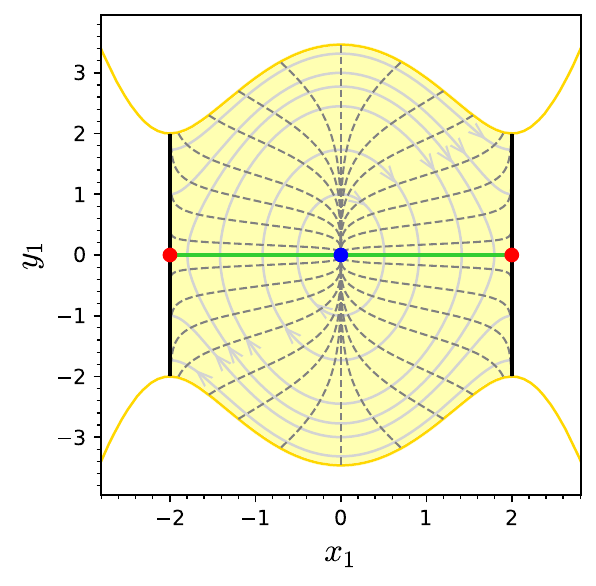}
  \label{fig:sub2d}
\end{subfigure}
\caption{ The chemical reaction model. Each leaf of the foliation is obtained as the pre-image of a gradient flow line of $V_1$. On the left (case I), $H^{-1}(E)$ admits a $3-2-3$ foliation, i.e., both sides $S_E=H^{-1}(E)\cap \{x_1\geq 0\}$ and $S_E':=H^{-1}(E) \cap \{x_1\leq 0\}$ admit a weakly convex foliation with $l=1$. On the right (case II),  $S_E= H^{-1}(E) \cap \{|x_1| \leq (\alpha/\beta)^{1/2}\}$ admits a weakly convex foliation with $l=2$.}
\label{fig:test3}
\end{figure}

\section*{Acknowledgement}

NdP was partially supported by CAPES/MATH-AMSUD 88881.878892/2023-01. SK was supported by the National Research Foundation of Korea (NRF) grant funded by the Korean government (MSIT) (No.\ RS-2025-16070003).  A part of this work was done during SK's visit to the Mathematisches Forschungsinstitut Oberwolfach (MFO) as an Oberwolfach Leibniz Fellow. SK cordially thanks the MFO for its excellent support and stimulating working atmosphere. PS acknowledges the support of the NYU-ECNU Institute of Mathematical Sciences at NYU Shanghai and the 2022 National Foreign Experts Program. 
PS was partially supported by FAPESP (2016/25053-8) and CNPq (306106/2016-7). PS was partially supported by the National Natural Science Foundation of China (grant number W2431007). PS thanks the support of the Shenzhen International Center for Mathematics - SUSTech. AS thanks the Instituto de Matemática Pura e Aplicada (IMPA) for the post-doc position. Part of this work was conducted during visits to the Southern University of Science and Technology (SUSTech) and the Kongju National University (KNU). AS thanks both institutes for their hospitality.

\bibliographystyle{plain}
\bibliography{bibliografia}

\begin{thebibliography}{10}

\bibitem{abbas2019holomorphic}
Casim Abbas and Helmut Hofer.
\newblock {\em Holomorphic curves and global questions in contact geometry}.
\newblock Birkh\"{a}user Advanced Texts: Basler Lehrb\"{u}cher. [Birkh\"{a}user
  Advanced Texts: Basel Textbooks]. Birkh\"{a}user/Springer, Cham, 2019.

\bibitem{albers2012global}
Peter Albers, Joel~W. Fish, Urs Frauenfelder, Helmut Hofer, and Otto van Koert.
\newblock Global surfaces of section in the planar restricted 3-body problem.
\newblock {\em Arch. Ration. Mech. Anal.}, 204(1):273--284, 2012.

\bibitem{bahni}
Yannis B\"ahni.
\newblock On a theorem by {S}chlenk.
\newblock {\em Calc. Var. Partial Differential Equations}, 63(5):Paper No. 124,
  19, 2024.

\bibitem{Barratt83}
C.~Barratt.
\newblock The {S}tark effect on an excited hydrogen atom.
\newblock {\em Am. J. Phys.}, 51(7):610--612, 1983.

\bibitem{de1998two}
J.~C. Bastos~de Figueiredo, C.~Grotta~Ragazzo, and C.~P. Malta.
\newblock Two important numbers in the {H}\'{e}non-{H}eiles dynamics.
\newblock {\em Phys. Lett. A}, 241(1-2):35--40, 1998.

\bibitem{churchill1979survey}
R.~C. Churchill, G.~Pecelli, and D.~L. Rod.
\newblock A survey of the {H}\'{e}non-{H}eiles {H}amiltonian with applications
  to related examples.
\newblock In {\em Stochastic behavior in classical and quantum {H}amiltonian
  systems ({V}olta {M}emorial {C}onf., {C}omo, 1977)}, volume~93 of {\em
  Lecture Notes in Phys.}, pages 76--136. Springer, Berlin-New York, 1979.

\bibitem{churchill1980pathology}
Richard~C. Churchill and David~L. Rod.
\newblock Pathology in dynamical systems. {III}. {A}nalytic {H}amiltonians.
\newblock {\em J. Differential Equations}, 37(1):23--38, 1980.

\bibitem{CFvK17}
Kai Cieliebak, Urs Frauenfelder, and Otto van Koert.
\newblock Periodic orbits in the restricted three-body problem and {A}rnold's
  {$J^+$}-invariant.
\newblock {\em Regul. Chaotic Dyn.}, 22(4):408--434, 2017.

\bibitem{CDR}
Vincent Colin, Pierre Dehornoy, and Ana Rechtman.
\newblock On the existence of supporting broken book decompositions for contact
  forms in dimension 3.
\newblock {\em Inventiones mathematicae}, 231(3):1489--1539, 2023.

\bibitem{Weakconvex}
Naiara~V. de~Paulo, Umberto~L. Hryniewicz, Seongchan Kim, and Pedro A.~S.
  Salom\~{a}o.
\newblock Genus zero transverse foliations for weakly convex {R}eeb flows on
  the tight $3$-sphere.
\newblock {\em Advances in Mathematics}, 457:109909, 2024.

\bibitem{dePaulo_Salomao}
Naiara~V. de~Paulo and Pedro A.~S. Salom\~{a}o.
\newblock Systems of transversal sections near critical energy levels of
  {H}amiltonian systems in {$\mathbb R^4$}.
\newblock {\em Mem. Amer. Math. Soc.}, 252(1202):v+105, 2018.

\bibitem{dePaulo_Salomao2}
Naiara~V. de~Paulo and Pedro A.~S. Salom\~{a}o.
\newblock On the multiplicity of periodic orbits and homoclinics near critical
  energy levels of {H}amiltonian systems in {$\mathbb R^4$}.
\newblock {\em Trans. Amer. Math. Soc.}, 372(2):859--887, 2019.

\bibitem{dPSsurvey}
Naiara~V. de~Paulo and Pedro A.~S. Salom\~ao.
\newblock Reeb flows, pseudo-holomorphic curves and transverse foliations.
\newblock {\em S\~ao Paulo J. Math. Sci.}, 16(1):314--339, 2022.

\bibitem{Euler1}
Leonhard Euler.
\newblock Un corps \'etant attir\'e en raison d\'eciproque quarr\'ee des
  distances vers deux points fixes donn\'es.
\newblock {\em M\'emoires de l'Acad. de Berlin}, pages 228--249, 1760.

\bibitem{Euler21}
Leonhard Euler.
\newblock De motu corporis ad duo centra virium fixa attracti.
\newblock {\em Novi Commentarii Academiae Scientiarum Imperialis
  Petropolitanae}, 10:207--242, 1766.

\bibitem{Euler22}
Leonhard Euler.
\newblock De motu corporis ad duo centra virium fixa attracti.
\newblock {\em Novi Commentarii Academiae Scientiarum Imperialis
  Petropolitanae}, 11:152--184, 1767.

\bibitem{FS}
J.~Fish and R.~Siefring.
\newblock Connected sums and finite energy foliations i: Contact connected
  sums.
\newblock {\em J. Symplectic Geom.}, 16(6):1639--1748, 2018.

\bibitem{Stark_Urs}
Urs Frauenfelder.
\newblock The {S}tark problem as a concave toric domain.
\newblock {\em Geom. Dedicata}, 217(1):Paper No. 10, 12, 2023.

\bibitem{FvKbook}
Urs Frauenfelder and Otto van Koert.
\newblock {\em The restricted three-body problem and holomorphic curves}.
\newblock Pathways in Mathematics. Birkh\"{a}user/Springer, Cham, 2018.

\bibitem{clodoaldo_e_pedro}
C.~Grotta-Ragazzo and Pedro A.~S. Salom\~{a}o.
\newblock The {C}onley-{Z}ehnder index and the saddle-center equilibrium.
\newblock {\em J. Differential Equations}, 220(1):259--278, 2006.

\bibitem{henon1964applicability}
Michel H\'{e}non and Carl Heiles.
\newblock The applicability of the third integral of motion: {S}ome numerical
  experiments.
\newblock {\em Astronom. J.}, 69:73--79, 1964.

\bibitem{Hezel92}
T.~P. Hezel, C.~E. Burkhardt, M.~Ciocca, and J.~J. Leventhan.
\newblock Classical view of the {S}tark effect in hydrogen atoms.
\newblock {\em Am. J. Phys.}, 60(4):324--328, 1992.

\bibitem{Hoferinitial}
H.~Hofer.
\newblock Pseudoholomorphic curves in symplectizations with applications to the
  {W}einstein conjecture in dimension three.
\newblock {\em Invent. Math.}, 114(3):515--563, 1993.

\bibitem{hofer1996characterisation}
H.~Hofer, K.~Wysocki, and E.~Zehnder.
\newblock {A characterisation of the tight three-sphere}.
\newblock {\em Duke J. Math}, 81(1):159--226, 1996.

\bibitem{properties_1}
H.~Hofer, K.~Wysocki, and E.~Zehnder.
\newblock Properties of pseudoholomorphic curves in symplectisations. {I}.
  {A}symptotics.
\newblock {\em Ann. Inst. H. Poincar\'{e} Anal. Non Lin\'{e}aire},
  13(3):337--379, 1996.

\bibitem{hofer1998dynamics}
H.~Hofer, K.~Wysocki, and E.~Zehnder.
\newblock The dynamics on three-dimensional strictly convex energy surfaces.
\newblock {\em Ann. of Math. (2)}, 148(1):197--289, 1998.

\bibitem{hhofer1999characterization}
H.~Hofer, K.~Wysocki, and E.~Zehnder.
\newblock {A characterization of the tight three sphere II}.
\newblock {\em Commun. Pure Appl. Anal}, 55:1139--1177, 1999.

\bibitem{properties_3}
H.~Hofer, K.~Wysocki, and E.~Zehnder.
\newblock {\em {Properties of Pseudoholomorphic Curves in Symplectizations III:
  Fredholm Theory}}, pages 381--475.
\newblock Birkh{\"a}user Basel, Basel, 1999.

\bibitem{hofer2003finite}
H.~Hofer, K.~Wysocki, and E.~Zehnder.
\newblock Finite energy foliations of tight three-spheres and {H}amiltonian
  dynamics.
\newblock {\em Ann. of Math. (2)}, 157(1):125--255, 2003.

\bibitem{Spatial}
X.~Hu, L.~Liu, Y.~Ou, Pedro A.~S. Salom\~ao, and G.~Yu.
\newblock A symplectic dynamics approach to the spatial isosceles three-body
  problem.
\newblock {\em to appear in J. Eur. Math. Soc. (JEMS)}.

\bibitem{Isayev72}
Y.~N. Isayev and A.~L. Kunitsyn.
\newblock To the problem of satellite's perturbed motion under the influence of
  solar radiation pressure.
\newblock {\em Celest. Mech.}, 1(6):44--51, 1972.

\bibitem{kim2024transverse}
Seongchan Kim.
\newblock Transverse foliations in the rotating {K}epler problem.
\newblock {\em J. Fixed Point Theory Appl.}, 26(1):Paper No. 4, 43, 2024.

\bibitem{Lantoine09}
G.~Lantoine and R.~P. Russell.
\newblock The {S}tark model: an exact, closed-form approach to low-thrust
  trajectory optimization.
\newblock In {\em 21st International Symposium on Space Flight Dynamics,
  Toulouse, France, Sep 28--Oct 2}, 2009.

\bibitem{Lemos}
C.~Lemos~de Oliveira.
\newblock 3-2-1 foliations for {R}eeb flows on the tight 3-sphere.
\newblock {\em Trans. Amer. Math. Soc.}, 377:3983--4053, 2024.

\bibitem{long2012index}
Yiming Long.
\newblock {\em Index theory for symplectic paths with applications}, volume 207
  of {\em Progress in Mathematics}.
\newblock Birkh\"{a}user Verlag, Basel, 2002.

\bibitem{lunsford1972stability}
G.~H. Lunsford and J.~Ford.
\newblock {On the stability of periodic orbits for nonlinear oscillator systems
  in regions exhibiting stochastic behavior}.
\newblock {\em Journal of Mathematical Physics}, 13(5):700--705, 1972.

\bibitem{Chemical}
W.~Lyu, S.~Naik, and S.~Wiggins.
\newblock Elementary exposition of realizing phase space structures relevant to
  chemical reaction dynamics.
\newblock {\em arXiv preprint}, 2020.

\bibitem{moser_coordinates}
J.~Moser.
\newblock On the generalization of a theorem of a. liapounoff.
\newblock {\em Comm. Pure Appl. Math.}, 11(2):257--271, 1958.

\bibitem{OM}
Diarmuid \'{O}~Math\'{u}na.
\newblock {\em Integrable systems in celestial mechanics}, volume~51 of {\em
  Progress in Mathematical Physics}.
\newblock Birkh\"{a}user Boston, Inc., Boston, MA, 2008.
\newblock With an appendix by Vincent G. Hart and Se\'{a}n Murray.

\bibitem{ragazzo1994nonintegrability}
C.~G. Ragazzo.
\newblock {Nonintegrability of some Hamiltonian systems, scattering and
  analytic continuation}.
\newblock {\em Communications in Mathematical Physics}, 166(2):255--277, 1994.

\bibitem{russmann_refinement}
H.~R\"ussmann.
\newblock Uber das verhalten analytischer hamiltonscher differentialgleichungen
  in der n\"ahe einer gleichgewichtsl\"osung (german).
\newblock {\em Math. Ann.}, 154(4):285--300, 1964.

\bibitem{salomao2004convex}
Pedro A.~S. Salom\~{a}o.
\newblock Convex energy levels of {H}amiltonian systems.
\newblock {\em Qual. Theory Dyn. Syst.}, 4(2):439--457 (2004), 2003.

\bibitem{alexsandro_lens}
A.~Schneider.
\newblock Global surfaces of section for dynamically convex {R}eeb flows on
  lens spaces.
\newblock {\em Trans. Amer. Math. Soc.}, 373(4):2775--2803, 2020.

\bibitem{Si3}
Richard Siefring.
\newblock Finite-energy pseudoholomorphic planes with multiple asymptotic
  limits.
\newblock {\em Math. Ann.}, 368(1-2):367--390, 2017.

\bibitem{Stark1914}
J.~Stark.
\newblock Beobachtungen {\"u}ber den {E}ffekt des elektrischen {F}eldes auf
  {S}pektrallinien {I}. {Q}uereffekt.
\newblock {\em Annalen der Physik}, 43(7):965--983, 1914.

\bibitem{WDR}
Holger Waalkens, Holger~R. Dullin, and Peter~H. Richter.
\newblock The problem of two fixed centers: bifurcations, actions, monodromy.
\newblock {\em Phys. D}, 196(3-4):265--310, 2004.

\bibitem{wendl2008finite}
Chris Wendl.
\newblock Finite energy foliations on overtwisted contact manifolds.
\newblock {\em Geometry \& Topology}, 12(1):531--616, 2008.

\bibitem{wendl2010automatic}
Chris Wendl.
\newblock Automatic transversality and orbifolds of punctured holomorphic
  curves in dimension four.
\newblock {\em Commentarii Mathematici Helvetici}, 85(2):347--407, 2010.

\end{thebibliography}

\end{document}